\documentclass[12pt, reqno, oneside]{amsart}
\usepackage{amssymb,amsthm,amsfonts,amsmath, amscd}

\usepackage{hyperref}
\usepackage{mathrsfs}
\usepackage{color}

\newcommand\Y{\mathbb Y}
\newcommand\Z{\mathbb Z}
\newcommand\C{\mathbb C}
\newcommand\R{\mathbb R}

\renewcommand\L{\mathbb L}

\newcommand\AAA{\mathbb A}

\newcommand\al{\alpha}
\newcommand\be{\beta}

\newcommand\Ga{\Gamma}
\newcommand\de{\delta}

\newcommand\ka{\varkappa}
\newcommand\La{\Lambda}

\newcommand\si{\sigma}
\newcommand\epsi{\varepsilon}

\newcommand\wt{\widetilde}
\newcommand\wh{\widehat}
\newcommand\half{\frac12}

\newcommand\Sign{\operatorname{Sign}}
\newcommand\discr{{\operatorname{discr}}}
\newcommand\Res{\operatorname{Res}}

\newcommand\FF{\mathcal F}
\newcommand\Aa{\mathcal A}
\newcommand\BB{\mathcal B}
\newcommand\CC{\mathcal C}
\newcommand\DD{\mathcal D}
\newcommand\CBD{\text{$\CC$-$\BB$-$\DD$}}

\renewcommand\L{L}

\newcommand\ccdot{\,\cdot\,}

\newtheorem{theorem}{Theorem}[section]
\newtheorem{proposition}[theorem] {Proposition}

\newtheorem{corollary}[theorem]{Corollary}
\newtheorem{lemma}[theorem]{Lemma}

\newtheorem{theoremA}{Theorem}

\theoremstyle{definition}
\newtheorem{definition}[theorem]{Definition}
\newtheorem{remark}[theorem]{Remark}

\newtheorem{example}[theorem]{Example}

\newtheorem*{problem}{Problem}

\numberwithin{equation}{section}

\begin{document}

\title[Characters and splines]{Characters of classical groups, Schur-type functions,  and discrete splines}

\author{Grigori Olshanski}

\begin{abstract}
We study a spectral problem related to the finite-dimensional characters of the groups $Sp(2N)$, $SO(2N+1)$, and $SO(2N)$, which  form the classical series $\CC$, $\BB$, and $\DD$, respectively. The irreducible characters of these three series are given by $N$-variate symmetric polynomials $\chi_{\nu,N}(x_1,\dots,x_N)$, where the index $\nu$ ranges over the set $\Sign^+_N$ of $N$-tuples of integers $\nu_1\ge\dots\ge\nu_N\ge0$ (the case of the series $\DD$ requires a minor correction). The spectral problem in question consists in the decomposition of the characters after their restriction to the subgroups of the same type but smaller rank $K<N$. More precisely, we are dealing with normalized characters and the expansion
$$
\frac{\chi_{\nu,N}(x_1,\dots,x_K,1,\dots,1)}{\chi_{\nu,N}(1,\dots,1)}
=\sum_{\ka\in\Sign^+_K}\La^N_K(\nu,\ka)\frac{\chi_{\ka,K}(x_1,\dots,x_K)}{\chi_{\ka,K}(1,\dots,1)}.
$$
The main result of the paper is the derivation of explicit determinantal formulas for the coefficients $\La^N_K(\nu,\ka)$.

In fact, we first compute these coefficients  in a greater generality --- for the multivariate symmetric Jacobi polynomials depending on two continuous parameters. Next, we show that the formulas can be drastically simplified for the three special cases of Jacobi polynomials corresponding to the $\CBD$ characters. In particular, we show that the coefficients $\La^N_K(\nu,\ka)$, thought of as functions of $\ka\in\Sign^+_K$, are given by piecewise polynomial functions. This is where a link with discrete splines arises. 
 
In type $\Aa$ (that is, for the characters of the unitary groups $U(N)$), similar results were earlier obtained by Alexei Borodin and the author [Adv. Math., 2012], and then reproved by another method by Leonid Petrov [Moscow Math. J., 2014].  The case of the symplectic and orthogonal characters is more intricate. 
\end{abstract}

\date{}

\thanks{Supported by the Russian Science Foundation under project 23-11-00150}

\maketitle

\tableofcontents

\section{Introduction}\label{sect1}

This introductory section is structured as follows. We begin with a brief description of the problem and its history (section \ref{sect1.1}) and a discussion of spline functions (sections \ref{sect1.2} --\ref{sect1.4}). Then we introduce a few necessary definitions (sections \ref{sect1.5}--\ref{sect1.6}). After that, in sections \ref{sect1.7}--\ref{sect1.9}, the main results are stated. Various comments and bibliographical notes are collected in section \ref{sect1.10}. 

\subsection{Stochastic matrices related to irreducible characters}\label{sect1.1}

Let $G$ be a finite or compact group and $\{\chi_{\nu,G}\}$ be the set of its irreducible characters, where the indices $\nu$ are appropriate labels. Let us regard the set $\{\chi_{\nu,G}\}$ (or simply the corresponding set $\{\nu\}$ of labels) as a kind of dual object $\wh{G}$ to the group $G$. Then one would like to assign to any morphism $\phi\colon H\to G$, a dual ``morphism'' $\wh\phi$ from $\wh{G}$ to $\wh{H}$; how to do this? A reasonable solution is as follows. It is more convenient to work with \emph{normalized} irreducible characters
$$
\wt\chi_{\nu,G}:=\frac{\chi_{\nu,G}}{\dim\nu}, \quad \dim\nu:=\chi_{\nu,G}(e).
$$
The pullback of $\wt\chi_{\nu,G}$ under $\phi$ is a normalized, positive definite class function on $H$, and hence it can be written as a convex combination of normalized irreducible characters $\wt\chi_{\ka,H}$ of the group $H$:
$$
\chi_{\nu,G}\circ\phi=\sum_\ka\La^G_H(\nu,\ka) \wt\chi_{\ka,H},
$$
where $\La^G_H(\nu,\ka)$ are some coefficients. These coefficients obviously form a stochastic matrix $\La^G_H$ of the format $\{\nu\}\times\{\ka\}$, and we regard $\La^G_H$ as the  desired dual ``morphism'' $\wh\phi$ from $\wh G$ to $\wh H$.  This is justified by the fact that  stochastic matrices (more generally, Markov kernels) can be viewed as a natural generalization of ordinary maps.\footnote{The idea to consider Markov kernels as morphisms also arises in other situations, see e.g. \cite{Lawvere}.} 

\begin{example}
Take $G=S(n)$, the symmetric group on the set $\{1,\dots,n\}$; then the corresponding set $\{\nu\}$ is $\Y_n$, the set of Young diagrams with $n$ boxes. Next, for $k<n$, take $H=S(k)$, the subgroup of $S(n)$ fixing the points $k+1,\dots, n$; the corresponding set $\{\ka\}$ is $\Y_k$. Then
\begin{equation}\label{eq1.L}
\La^{S(n)}_{S(k)}(\nu,\ka)=\begin{cases}\dfrac{\dim\ka\, \dim\nu/\ka}{\dim\nu}, &\text{if $\ka\subset\nu$,}\\
 0, &\text{otherwise,}
\end{cases} 
\end{equation}
where $\dim (\,\cdot\,)$ is the number of standard tableaux of a given (skew) shape.
\end{example} 

For this quantity one can obtain a determinantal expression which can be transformed to the following form: 
\begin{equation}\label{eq1.M}
\La^{S(n)}_{S(k)}(\nu,\ka)=\frac{\dim\ka}{n^{\downarrow k}}\,  s^*_\ka(\nu_1,\nu_2,\dots),
\end{equation}
where 
$$
n^{\downarrow k}:=n(n-1)\dots(n-k+1)
$$
and $s^*_\ka$ is the \emph{shifted Schur function} indexed by $\ka$; these functions form a basis of the \emph{algebra of shifted symmetric functions},  see \cite[Theorem 8.1]{OO-AA}, \cite[Prop. 6.5]{BO-book}.

Formula \eqref{eq1.M} makes it possible to find the asymptotics of $\La^{S(n)}_{S(k)}(\nu,\ka)$ for fixed $\ka$ and growing $\nu$. As an application, one obtains a relatively simple proof of Thoma's theorem about characters of the infinite symmetric group: \cite{KOO}, \cite{BO-book}.  

Another notable fact is that the quantity 
$$
\frac{n^{\downarrow k}}{\dim\ka}\,\La^{S(n)}_{S(k)}(\nu,\ka),
$$
viewed as a function of the partition $\nu$, has a similarity with the Schur function (which is not evident from the initial definition \eqref{eq1.L}). 

In \cite{BO-AdvMath}, we raised the problem of studying the stochastic matrices $\La^{U(N)}_{U(K)}$ related to the unitary group characters; here the matrix entries $\La^{U(N)}_{U(K)}(\nu,\ka)$ are indexed by tuples of integers
$$
\nu=(\nu_1\ge\dots\ge\nu_N), \quad \ka=(\ka_1\ge\dots\ge\ka_K), \quad K<N.
$$
In that work we were guided by a remarkable analogy between the infinite symmetric group $S(\infty)$ and the infinite-dimensional unitary group $U(\infty)$.\footnote{For more detail about this analogy, see \cite{BO-MMJ}. About applications of the stochastic matrices $\La^{U(N)}_{U(K)}$, see the expository paper \cite{Ols-ICM}. }

We obtained again a determinantal ``Schur-type'' expression for the matrix entries $\La^{U(N)}_{U(K)}(\nu,\ka)$ and could  apply it to a novel derivation of the classification theorem for the characters of the infinite-dimensional unitary group $U(\infty)$.  

The aim of the present paper is to extend the results of \cite{BO-AdvMath} to other series of compact classical groups, that is, the symplectic groups $Sp(2N)$ (series $\CC$) and the orthogonal groups $SO(2N+1)$ and $SO(2N)$ (series $\BB$ and $\DD$).  As it often happens in representation theory, working with the $\CBD$ series turns out to be harder than with series $\Aa$. The present paper is focused on the combinatorial aspects of the problem, and the asymptotic analysis is postponed to a separate publication.

\subsection{B-splines}\label{sect1.2}

Given an $N$-tuple of real numbers $y_1>\dots>y_N$, we define a function of a variable $x\in\R$ by
\begin{equation}\label{eq1.A}
M_N(x;y_1,\dots,y_N):=(N-1)\sum_{i:\, y_i\ge x}\frac{(y_i-x)^{N-2}}{\prod\limits_{r:\, r\ne i}(y_i-y_r)}.
\end{equation}
Note that the number of terms on the right-hand side depends on the position of the variable $x$ relative to the parameters $y_1,\dots,y_N$.  The function $x\mapsto M_N(x;y_1,\dots,y_N)$ has the following properties:

\begin{itemize}

\item[(i)] it vanishes outside the interval $(y_N,y_1)$;

\item[(ii)]  it is piecewise polynomial: on each interval $(y_{i+1},y_i)$ inside $(y_N,y_1)$, it is given by a polynomial of degree $N-2$;

\item[(iii)] it has $N-3$ continuous derivatives at each point $y_i$;

\item[(iv)] it is positive on $(y_N,y_1)$ and 
$$
\int_{y_N}^{y_1}M_N(x;y_1,\dots,y_N)dx=1,
$$
so that $M_N(x;y_1,\dots,y_N)dx$ is a compactly supported probability measure on $\R$.

\end{itemize}

The function $x\mapsto M_N(x;y_1,\dots,y_N)$ is called the \emph{B-spline} with the knots $y_1,\dots,y_N$ (``B''  is an abbreviation of  ``basic'').   The paper \cite{CS} by Curry and Schoenberg contains remarkable results about the B-spline. For more details about spline functions we refer to Schumaker's monograph \cite{Schumaker}. 

The proposition below relates B-splines to random matrices. Let $H(N)$ be the space of $N\times N$ Hermitian matrices. The unitary group $U(N)$ acts on $H(N)$ by conjugations. Let $\mathcal O(y_1,\dots,y_N)\subset H(N)$ denote the set of matrices with eigenvalues $y_1,\dots,y_N$. It is an $U(N)$-orbit and carries a (unique) $U(N)$-invariant probability measure, which we denote by $P(y_1,\dots,y_N)$. 

\begin{proposition}[Okounkov]
Consider the projection $\mathcal O(y_1,\dots,y_N)\to\R$ assigning to a matrix $X\in \mathcal O(y_1,\dots,y_N)$ its upper left most entry $X_{11}$. The pushforward of the measure $P(y_1,\dots,y_N)$ under this projection 
is the measure $M_N(x;y_1,\dots,y_N)dx$.  
\end{proposition}

As noted by Okounkov (\cite[Remark 8.2]{OV}), this is an easy corollary of \cite[Theorem 2]{CS}.

\subsection{Discrete B-splines}\label{sect1.3}

Throughout the paper we use the standard notation for the Pochhammer symbol (aka raising factorial power):
$$
(x)_m:=x(x+1)\dots(x+m-1)=\frac{\Ga(x+m)}{\Ga(x)}, \quad m=0,1,2,\dots\,.
$$

By the \emph{discrete B-spline} with integral knots $y_1>\dots>y_N$ we mean the function on $\Z$ defined by
\begin{equation}\label{eq1.B}
M^\discr_N(x;y_1,\dots,y_N):=(N-1)\sum_{i:\, y_i\ge x}\frac{(y_i-x+1)_{N-2}}{\prod\limits_{r:\, r\ne i}(y_i-y_r)}.
\end{equation}
This agrees with the definition in Schumaker \cite[section 8.5]{Schumaker}, up to minor changes. 
Note that the right-hand side of \eqref{eq1.B} is not affected if instead of $y_i\ge x$ we impose the weaker condition $y_i+N-2\ge x$: the reason is that the function $x\mapsto (y-x+1)_{N-2}$ vanishes at the points $y+1,\dots,y+N-2$. The formula \eqref{eq1.B} is very similar to \eqref{eq1.A}, only  the variable $x$ now ranges over $\Z$ and not $\R$, and the ordinary powers are replaced by the raising factorial powers. The discrete B-spline has properties similar to properties (i)--(iv) above. In particular, it determines a finitely supported probability measure on $\Z$ (the support is the lattice interval $y_N+N-2\le x\le y_1$). 

\subsection{Link with characters of $U(N)$}\label{sect1.4}

Let $\Sign_N\subset \Z^N$ denote the set of $N$-tuples of integers $\nu=(\nu_1,\dots,\nu_N)$ subject to the inequalities $\nu_1\ge\dots\ge\nu_N$. The elements $\nu\in\Sign_N$ are called \emph{signatures of length $N$}.  They parametrize the irreducible characters of the unitary group $U(N)$; these characters  can be thought of as symmetric $N$-variate Laurent polynomials (aka rational Schur functions) and are denoted as $\chi_{\nu,N}(u_1,\dots,u_N)$. We also introduce the \emph{normalized} characters 
\begin{equation}\label{eq1.C}
\wt\chi_{\nu,N}(u_1,\dots,u_N):=\frac{\chi_{\nu,N}(u_1,\dots,u_N)}{\chi_{\nu,N}(1,\dots,1)}, \qquad \nu\in\Sign_N.
\end{equation}

Observe that $u\mapsto \wt\chi_{\nu,N}(u,1,\dots,1)$ is a univariate Laurent polynomial and consider its expansion on monomials $u^k$, which we write in the form 
\begin{equation}\label{eq1.D}
\wt\chi_{\nu,N}(u,1,\dots,1)=\sum_{k\in\Z}\La^N_1(\nu,k) u^k.
\end{equation}

The coefficients $\La^N_1(\nu,k)$ are nonnegative real numbers which sum to $1$, for any fixed $\nu\in\Sign_N$. Thus, we may view $\La^N_1(\nu,\ccdot)$ as a finitely supported probability distribution on $\Z$.

In combinatorial terms, the quantity $\chi_{\nu,N}(1,\dots,1)$ (the dimension of the character $\chi_{\nu,N}$) is equal to the number of triangular Gelfand-Tsetlin patterns with the top row $\nu$, whereas $\La^N_1(\nu,k)$ is the fraction of the patterns with the bottom entry  $k$.

\begin{proposition}[\cite{BO-AdvMath}, formula (7.10)]
For any signature $\nu\in\Sign_N$, the  distribution $\La^N_1(\nu,\ccdot)$ is the discrete B-spline with the knots $y_i=\nu_i-i+1${\rm:}
\begin{equation}\label{eq1.E}
\La^N_1(\nu,k)=M_N^\discr(k; \nu_1,\nu_2-1,\dots,\nu_N-N+1), \qquad k\in\Z.
\end{equation}
\end{proposition}

Formally, \cite[(7.10)]{BO-AdvMath} assumes $k\ge1$; however, the whole picture is invariant under the simultaneous shift of all coordinates by an arbitrary integer, so that this constraint can be dropped. Another derivation of \eqref{eq1.E} can be obtained from the proof of Petrov \cite[Theorem 1.2]{Petrov}.

\subsection{Schur-type functions}\label{sect1.5}

Let $\Sign^+_N\subset\Sign_N$ denote the set of \emph{positive signatures of length $N$}: these are $N$-tuples of integers $\nu=(\nu_1,\dots,\nu_N)$ subject to the constraints $\nu_1\ge\dots\ge \nu_N\ge0$.

Let $\phi_0(x)\equiv1, \phi_1(x), \phi_2(x),\dots$ be an infinite sequence of
functions of a variable $x$. For a signature $\nu\in\Sign^+_N$ we define
a symmetric function of $N$ variables $x_1,\dots,x_N$ by
\begin{equation}\label{eq1.F}
\phi_{\nu, N}(x_1,\dots,x_N):=\dfrac{\det[\phi_{\nu_i+N-i}(x_j)]_{i,j=1}^N}
{\det[\phi_{N-i}(x_j)]_{i,j=1}^N}.
\end{equation}
Note that $\phi_{\emptyset, N}(x_1,\dots,x_N)\equiv1$, where $\emptyset:=(0,\dots,0)$.

The definition \eqref{eq1.F} fits into the formalism suggested by Nakagawa, Noumi, Shirakawa, and Yamada \cite{NNSY} as an alternative approach to Macdonald's \emph{9th variation of Schur functions} \cite{Mac-SLC}. The latter term is historically justified, but somewhat inconvenient to use.   For this reason we prefer to call the functions \eqref{eq1.F} the \emph{Schur-type functions}. 

If $\phi_n$ is a polynomial of degree $n$ ($n=0,1,2,\dots$), then the denominator on the right-hand side  is proportional to the Vandermonde determinant
$$
V(x_1,\dots,x_N):=\prod_{1\le i<j\le N}(x_i-x_j),
$$
which implies that the functions $\phi_{\nu,N}$ are symmetric polynomials. Such Schur-type functions are sometimes called  \emph{generalized Schur polynomials}, see e.g. Sergeev and Veselov \cite{SV}. Note that they form a basis in the algebra of $N$-variate symmetric polynomials. In the particular case when $\phi_n(x)=x^n$ ($n=0,1,2,\dots$) we obtain the ordinary Schur polynomials.  

Throughout the present paper, various Schur-type functions $\phi_{\nu,N}$ will appear. Sometimes they will be polynomials, sometimes they won't.  

\subsection{Stochastic matrices $\La^N_K$ related to Jacobi polynomials}\label{sect1.6}
We are mainly interested in the characters of the compact classical groups 
\begin{equation}\label{eq1.G}
Sp(2N) \quad (\text{series $\CC$}), \quad SO(2N+1) \quad (\text{series $\BB$}), \quad SO(2N) \quad (\text{series $\DD$}).
\end{equation}
However, a substantial part of our results hold true in the broader context of multivariate Jacobi polynomials. 

Recall that the \emph{classical Jacobi polynomials} $P^{(a,b}_n(x)$  are the orthogonal polynomials with the weight function $(1-x)^a(1+x)^b$ on $[-1,1]$(Szeg\"o \cite{Szego}). The corresponding \emph{$N$-variate Jacobi polynomials} are defined by 
\begin{equation*}
P^{(a,b)}_{\nu,N}(x_1,\dots,x_N):=\frac{\det[P^{(a,b)}_{\nu_i+N-i}(x_j)]_{i,j=1}^N}{V(x_1,\dots,x_N)}, \qquad \nu\in\Sign^+_N.
\end{equation*}
These polynomials are an instance of generalized Schur polynomials (up to constant factors); they are also a particular case of the more general 3-parameter family of orthogonal polynomials associated with the root system $BC_N$ (see e.g. Lassale \cite{Lassalle} or Heckman's lectures in \cite{HS}).
 
The three distinguished cases of the Jacobi parameters 
\begin{equation}\label{eq1.CBD}
(a,b)=(\tfrac12,\tfrac12),\; (\tfrac12,-\tfrac12),\;(-\tfrac12,-\tfrac12)
\end{equation}
correspond to characters of the groups \eqref{eq1.G} (in the same order). More precisely, set $x_i=\tfrac12(u_i+u^{-1}_i)$ and regard $u_1^{\pm1},\dots,u^{\pm1}_N$ as the matrix eigenvalues. Then the polynomials $P^{(a,b)}_{\nu,N}$, suitably renormalized, turn into the irreducible characters, with the understanding that in the case of the series $\DD$ and $\nu_N>0$ one should take the sum of two ``twin'' irreducible characters (see e.g. Okounkov and Olshanski \cite{OO-KirSem}).

Constant factors may be neglected here, because we will deal with the normalized characters and the normalized polynomials
\begin{equation}\label{eq1.K}
\wt P^{(a,b)}_{\nu, N}(x_1,\dots,x_N):=\frac{P^{(a,b)}_{\nu, N}(x_1,\dots,x_N)}{P^{(a,b)}_{\nu, N}(1,\dots,1)}.
\end{equation}

\begin{definition}\label{def1.A}
With each couple $(N,K)$ of natural numbers $N>K\ge1$ we associate a matrix $\La^N_K$ of the format $\Sign^+_N\times\Sign^+_K$: the matrix entries $\La^N_K(\nu,\ka)$ are the coefficients in the expansion
\begin{equation}\label{eq1.H}
\wt P^{(a,b)}_{\nu, N}(x_1,\dots,x_K, 1,\dots,1)=\sum_{\ka\in\Sign^+_K}\La^N_K(\nu,\ka)\wt P^{(a,b)}_{\ka, K}(x_1,\dots,x_K).
\end{equation}
The matrix depends on the Jacobi parameters $(a,b)$, but we suppress them to simplify the notation. Our assumptions on the Jacobi parameters are the following:
\begin{equation}\label{eq1.Jacobi}
a>-1, \quad b>-1, \quad a+b\ge-1.
\end{equation}
\end{definition}

Here the first two inequalities ensure the integrability of the weight function. The third inequality is an additional technical assumption; it is obviously satisfied for the three special values \eqref{eq1.CBD}.

Our goal is to find explicit formulas for the quantities $\La^N_K(\nu,\ka)$, with the emphasis on the three distinguished cases \eqref{eq1.CBD} corresponding to the $\CBD$ characters. 

One may think of  $\La^N_K(\nu,\ka)$ as a function of the variable $\nu\in\Sign^+_N$, with $\ka\in\Sign^+_K$ being an index. Or, conversely, as a function of $\ka$ indexed by $\nu$. The first viewpoint is motivated by asymptotic representation theory, where one is interested in large-$N$ limits (Okounkov and Olshanski \cite{OO-Jack}, \cite{OO-Jacobi}). The second viewpoint has its origin in spectral problems of classical representation theory and leads, in the distinguished cases \eqref{eq1.CBD}, to multidimensional discrete splines. 

From the branching rule of the multivariate Jacobi polynomials (\cite[Proposition 7.5]{OO-Jacobi}) and the condition $a+b\ge-1$ it follows that the coefficients $\La^N_K(\nu,\ka)$ are nonnegative (in the three special cases \eqref{eq1.CBD} this also follows from the classical branching rule of the symplectic and orthogonal characters, see Zhelobenko \cite{Zhelobenko}). Next, the row sums of the matrix entries are equal to $1$ (to see this, substitute $x_1=\dots=x_K=1$ in \eqref{eq1.H}). This means that $\La^N_K(\nu,\ccdot)$ is a probability distribution on the set $\Sign^+_K$, for any fixed $\nu\in\Sign^+_N$. In other words, $\La^N_K$ is a stochastic matrix, and its entries $\La^N_K(\nu,\ka)$ may be viewed as transition probabilities between the sets $\Sign^+_N$ and $\Sign^+_K$. 

We proceed to the description of the main results.

\subsection{A Cauchy-type identity involving $\La^N_K$ (Theorem A)}\label{sect1.7}

Throughout the paper we use the notation
\begin{equation*}
L:=N-K+1
\end{equation*} 
and
\begin{equation*}
\epsi:=\frac{a+b+1}2.
\end{equation*}
We often use the parameters $(a,\epsi)$ instead of $(a,b)$. 

For a positive integer $N$ and $\nu\in\Sign^+_N$ we set
\begin{equation}\label{eq1.F_N}
F_N(t;\nu;\epsi):=\prod_{i=1}^N\frac{t^2-(N-i+\epsi)^2}{t^2-(\nu_i+N-i+\epsi)^2}.
\end{equation}
This is an even rational function of $t$. We call it the \emph{characteristic function} of the signature $\nu$. 

We also set
\begin{equation}\label{eq1.d_N}
d_N(\nu;\epsi):=\prod_{1\le i<j\le N}((\nu_i+N-i+\epsi)^2-(\nu_j+N-j+\epsi)^2).
\end{equation}
We need $d_N(\nu,\epsi)$ to be nonzero. Because of this (and for some other reasons) we impose the additional constraint $a+b\ge-1$, meaning that $\epsi\ge0$. This will guarantee $d_N(\nu,\epsi)\ne0$. In the distinguished cases \eqref{eq1.CBD} we have $\epsi=1,\frac12,0$. 

Next, we introduce a sequence of functions 
\begin{equation}\label{eq1.g_k}
g_k(t)=g_k(t;a,\epsi,L):={}_4F_3\left[\begin{matrix}-k,\, k+2\epsi,\, L,\, L+a\\-t+L+\epsi,\,
t+L+\epsi,\, a+1\end{matrix}\Biggl|1\right], \quad k=0,1,2,\dots\,.
\end{equation}
The right-hand side is a balanced (=Saalsch\"utzian) hypergeometric series (Bailey \cite[section 2.5]{Bailey}). Because $k$ is a nonnegative integer, the series terminates and represents a rational function of variable $t$. Because of the symmetry  $g_k(t)=g_k(-t)$, it is actually a rational function of $t^2$. Note that $g_0(t)\equiv1$. 

Note also that in a limit transition as $L\to\infty$, combined with a change of the variable $t$, the functions $g_k(t)$ degenerate into the Jacobi polynomials, see section \ref{sect9.3}. 

From the sequence $\{g_k(t)\}$ we form the Schur-type functions according to \eqref{eq1.F}:
\begin{equation}\label{eq1.Gkappa}
G_{\ka,K}(t_1,\dots,t_K)=\frac{\det[g_{\ka_i+K-i}(t_j)]_{i,j=1}^K}{\det[g_{K-i}(t_j)]_{i,j=1}^K}, \quad \ka\in\Sign^+_K\,.
\end{equation}

\begin{theoremA}[see Theorem \ref{thm4.A}]\label{thmA}
The following identity holds true
\begin{equation*}
\prod_{j=1}^KF_N(t_j;\nu;\epsi)=\sum_{\ka\in\Sign^+_K}\frac{\La^N_K(\nu,\ka)}{d_K(\ka;\epsi)}\,G_{\ka,K}(t_1,\dots,t_K).
\end{equation*}
The sum on the right-hand side is finite, and the quantities $\La^N_K(\nu,\ka)$ are uniquely determined by this formula. 
\end{theoremA}

This result has the form of the Cauchy identity connecting two families of multivariate functions, both indexed by the elements $\ka\in\Sign^+_K$. Specifically, these functions are $\nu\mapsto \La^N_K(\nu,\ka)/d_K(\ka;\epsi)$ and  $G_{\ka,K}(t_1,\dots,t_K)$. In the first family, we take as the variables the shifted coordinates $n_i:=\nu_i+N-i$, where $i=1,\dots,N$. Then on the other side of the identity we get a double product over the two sets of variables, 
$$
\prod_{i=1}^N\prod_{j=1}^K \frac{t_j^2-(N-i+\epsi)^2}{t_j^2-(n_i+\epsi)^2},
$$
which is separately symmetric with respect to the permutations of $n_1,\dots,n_N$ and $t_1,\dots,t_K$ --- just as in the classical Cauchy identity. 

\subsection{Determinantal formula for the matrix entries $\La^N_K(\nu,\ka)$ (Theorem B)}\label{sect1.8}

To state the result we need a few definitions. 
Given an integer $L\ge2$, we consider the infinite grid 
\begin{equation}\label{eq1.N}
\AAA(\epsi,L):=\{A_1,A_2,\dots\}\subset \R_{>0}, \qquad A_m:=L+\epsi+m-1, \quad m=1,2,\dots,
\end{equation}
and we denote by $\FF(\epsi,L)$ the vector space whose elements are even rational functions $f(t)$ of a complex variable $t$, which are regular at $t=\infty$ and such that their only singularities are simple poles contained in the set $(-\AAA(\epsi,L))\cup \AAA(\epsi,L)$. 
Obviously, 
$$
\FF(\epsi,2)\supset \FF(\epsi,3)\supset\dots,
$$
and all these spaces have countable dimension. We show that the functions $g_k(t)=g_k(t;a,\epsi,L)$ form a basis of $\FF(\epsi, L)$. Given $\phi\in\FF(\epsi,L)$, we denote by $(\phi:g_k)$ the coefficients in the expansion
\begin{equation}\label{eq1.phi}
\phi(t)=\sum_{k=0}^\infty (\phi:g_k) g_k(t).
\end{equation}

\begin{theoremA}[see Theorem \ref{thm4.B}]\label{thmB}
With the notation introduced above we have the following determinantal formula
\begin{equation}\label{eq1.detformula}
\frac{\La^N_K(\nu,\ka)}{d_K(\ka;\epsi)}=\det\left[(g_{K-j}F_N: g_{\ka_i+K-i})\right]_{i,j=1}^K.
\end{equation}
\end{theoremA} 

Note that the characteristic function \eqref{eq1.F_N} lies in the space $\FF(\epsi,N)$.  More generally, under our assumption that $L=N-K+1$, the function $g_{K-j}F_N$ lies in $\FF(\epsi,L)$ for any $j=1,\dots, K$. This  implies that the quantities $(g_{K-j}F_N: g_{\ka_i+K-i})$ are well defined. . 

The determinantal formula \eqref{eq1.detformula} resembles the classical Jacobi-Trudi formula for the Schur functions (or rather its version for the Macdonald's 9th variation of Schur functions). This result is deduced from Theorem A in the same way as the classical Jacobi-Trudi is deduced from the Cauchy identity. 

Theorem \ref{thmA} and Theorem \ref{thmB} show that if we treat $\nu$ as a variable and $\ka$ as a parameter, then the functions 
$$
\nu\mapsto \frac{\La^N_K(\nu,\ka)}{d_K(\ka;\epsi)}
$$
share two fundamental properties of Schur-type functions: Cauchy-type formula and Jacobi-Trudi-type formulas. 

\subsection{Computation of the matrix entries $\La^N_K(\nu,\ka)$ (Theorems \ref{thmC} and \ref{thmD})}\label{sect1.9}

Our subsequent actions are driven by the desire to find an explicit expression for the entries of the $K\times K$ matrix on the right-hand side of the formula \eqref{eq1.detformula}. This leads us to the problem of computing  the coefficients $(\phi:g_k)$ of the expansion \eqref{eq1.phi} for a given function  $\phi\in\FF(\epsi,L)$. 
Solving this problem will allow us find explicitly the matrix entries  $\La^N_K(\nu,\ka)$ by means of the determinantal formula \eqref{eq1.detformula}, because the matrix entries on the right-hand side are of the form $(\phi:g_k)$ with $\phi=g_{K-j}F_N\in\FF(\epsi,L)$ and  $k=\ka_i+K-i$. 

Our approach to the problem is the following. For a function $\phi\in \FF(\epsi,L)$, we denote by  $\Res_{t=A_m}{\phi(t)}$ its residue of $\phi(t)$ at a given point $t=A_m\in\AAA(\epsi,L)$. Because $\phi(t)$ is rational, it has finitely many nonzero poles only. We need the most natural and simplest  basis of the space $\FF(\epsi,L)$, which is formed by the functions 
\begin{equation}\label{eq1.e}
e_0(t)\equiv1, \qquad e_m(t):=\frac1{t-A_m}-\frac1{t+A_m}, \quad m\in\Z_{\ge1}.
\end{equation}
In accordance with the general notation \eqref{eq1.phi}, we denote by $(e_m:g_k)$ the transition coefficients between the bases $\{e_m\}$ and $\{g_k\}$. 

Next, it is not difficult to show  that for any  $\phi\in \FF(\epsi,L)$, 
\begin{equation}\label{eq1.I}
(\phi:g_k)=\begin{cases} \sum_{m\ge k}\Res_{t=A_m}(\phi(t))(e_m:g_k), & k\ge1, \\
\phi(\infty)+\sum_{m\ge1}\Res_{t=A_m}(\phi(t))(e_m: g_0), & k=0
\end{cases}
\end{equation}
(see Proposition \ref{prop5.A}). The series in \eqref{eq1.I} in fact terminates, because the number of poles is finite.

\begin{theoremA}[see Theorem \ref{thm5.A}]\label{thmC}
For the transition coefficients $(e_m:g_k)$ there is an explicit expression through a terminating hypergeometric series of type ${}_4F_3$.
\end{theoremA}

Combining \eqref{eq1.I} with Theorem \ref{thmC} we obtain an expression for the matrix entries on the right-hand of \eqref{eq1.detformula}. 
The final result looks complicated, because it involves the residues of the functions $\phi=g_{K-j}F_N$, which are given by certain hypergeometric series of type ${}_4F_3$, and also the transition coefficients, which are given by some other ${}_4F_3$ series. 

However, the situation radically simplifies for the symplectic and orthogonal characters. 

\begin{theoremA}\label{thmD}
Consider the three distinguished cases \eqref{eq1.CBD} of Jacobi parameters corresponding to the characters of the classical groups of type $\CBD$. Then the matrix entries $(g_{K-j}F_N: g_{\ka_i+K-i})$ on the right-hand side of \eqref{eq1.detformula} admit an explicit elementary expression. 
\end{theoremA}

A detailed formulation of this result is given in Theorem \ref{thm8.A}. It turns out that in the three distinguished cases, the two families of ${}_4F_3$ series (for the functions $g_k(t)$ and for the coefficients $(e_m:g_k)$) are miraculously summed up explicitly. This is shown in Theorem \ref{thm6.A} and Theorem \ref{thm7.A}, respectively.  

In the particular case $K=1$ we obtain symplectic and orthogonal versions of the discrete B-spline $\CBD$ (sect. \ref{sect8.2}).

\subsection{Notes}\label{sect1.10}

1. The present paper is a continuation of the work \cite{BO-AdvMath} by Borodin and the author. In \cite{BO-AdvMath}, similar results were obtained in type $\Aa$, that is, for the characters of the unitary groups $U(N)$. However, the case of symplectic and orthogonal characters, and especially that of multivariate Jacobi polynomials, is more difficult. 

2. Part of the results of \cite{BO-AdvMath} was reproved and extended by Petrov \cite{Petrov}. His method is very different; it allowed to compute the correlation kernel of a two-dimensional determinantal point process generated by the stochastic matrices $\La^N_{N-1}$ related to the characters of the unitary groups. The explicit expression for the matrix elements of $\La^N_K$ from \cite{BO-AdvMath} is then obtained as a direct corollary. Moreover, Petrov also obtained a $q$-version of these results. On the other hand, Petrov's approach does not give a Cauchy-type identity. 

3. In a scaling limit, the stochastic matrices $\La^N_K$ of all four types $\Aa$, $\BB$, $\CC$, $\DD$ degenerate to certain continuous Markov kernels, which are related to the corner processes from random matrix theory. These Markov kernels are given by determinantal expressions involving continuous spline functions. See  the author  \cite{Ols-JLT}, Faraut \cite{Faraut}, and Zubov \cite{Zubov}. The results of these works do not rely on \cite{BO-AdvMath}. On the other hand, they can be derived from the earlier results of Defosseux  \cite{Defosseux} about the correlation functions of corner processes.

4. The restriction problem for characters of classical groups and multivariate Jacobi polynomials was also considered by Gorin and Panova \cite{GP}, but from a different point of view. Namely, these authors were interested in finding explicit formulas for the resulting functions, and did not deal with their spectral expansion. Our approach describes the dual picture, related to that of \cite{GP} by a Fourier-type  transform. This reveals such  aspects of the problem, as the Cauchy-type identity from Theorem \ref{thmA} or the connection with discrete splines, which do not arise in the context of \cite{GP}. It seems to me that both approaches complement each other well. In the case of the unitary group characters, it is not too difficult (at least, for $K=1$) to derive the formulas of \cite{BO-AdvMath} from those of \cite{GP}, but in the case of the $\CBD$ characters this does not seem to be an easy task. 

5. The present work, as well as \cite{BO-AdvMath}, originated from a problem of asymptotic representation theory. 
Our formulas for the matrix elements $\La^N_K(\nu, \ka)$ in the case of the $\CBD$ characters are well suited for performing the large-$N$ limit transition in the spirit of \cite[section 8]{BO-AdvMath}. \footnote{I decided to postpone this material for a separate publication so as not to increase the size of the present paper. In connection with this topic, see also  section \ref{sect9.1} below.} This leads to one more approach to the classification of the extremal characters of the infinite-dimensional symplectic and orthogonal groups. Earlier works on this subject are Boyer \cite{Boyer}, Pickrell \cite{Pickrell}, Okounkov--Olshanski \cite{OO-Jack}, and Gorin--Panova \cite{GP}. 

6. An aspect of the present work, which seems to be of interest, is its connection with classical analysis. Such connections already arose in various problems concerning representations of infinite-dimensional groups. Here are some examples:

\begin{itemize}

\item The extremal spherical functions on the motion group of the real Hilbert space are given by the completely monotone functions (Schoenberg \cite{Sch-AnnMath}, the author \cite{Ols-1978}). 

\item Let $U(\infty)=\varinjlim U(N)$ be the direct limit (or simply the union) of the unitary groups $U(N)$  and let $H(\infty)=\varinjlim H(N)$ be the direct limit of the spaces of $N\times N$ Hermitian matrices. The extremal spherical functions on the semidirect product $U(\infty)\ltimes H(\infty)$ are parametrized by the totally positive functions on the real line (Schoenberg  \cite{Sch-JAM}, Pickrell \cite{Pickrell}, the author and Vershik \cite{OV}). 

\item The extremal characters of the infinite symmetric group $S(\infty)=\varinjlim S(n)$ (Thoma \cite{Thoma}, Vershik and Kerov \cite{VK-1981}) are in a one-to-one correspondence with the totally positive infinite triangular Toeplitz matrices, initially studied by Aissen, Schoenberg, and Whitney \cite{ASW}, Edrei \cite{Edrei-JAM}. 

\item Likewise, the extremal characters of the group $U(\infty)$ (Voiculescu \cite{Voiculescu}, Vershik and Kerov \cite{VK-1982}, Boyer \cite{Boyer}) correspond to more general totally positive infinite Toeplitz matrices (Edrei \cite{Edrei-TAMS}). 

\end{itemize}

The link with the B-spline and its discrete version adds one more item to this list. Note that the large-$N$ limit transition for the B-spline, studied by Curry and Schoenberg \cite{CS}, is directly connected with the asymptotic approach to the classification of spherical functions for $U(\infty)\ltimes H(\infty)$. Likewise, a similar asymptotic problem for the discrete B-spline is connected with the classification of characters of $U(\infty)$. 

Next, not so long ago, spline theorists also came up with a $q$-deformation of the B-spline: the first paper on this topic is Simeonov and Goldman \cite{SG}; from subsequent works on this topic, note Budak\c{c}i and Oru\c{c} \cite{Budakci}. As pointed out in the author's paper \cite{Ols2016}, this new version also arises in a representation-theoretic context related to the works Gorin \cite{Gorin}, Petrov \cite{Petrov}, Gorin and the author \cite{GO}. 

In a different direction, note an example of biorthogonal system discussed in section \ref{sect9.2}.

\subsection{Organization of the paper}

The short sections \ref{sect2} and \ref{sect3} contain a preparatory material. Then we proceed to proofs of Theorems \ref{thmA} and \ref{thmB} (section \ref{sect4}) and of Theorem \ref{thmC} (section \ref{sect5}). Sections \ref{sect6} and \ref{sect7} are devoted to simplification of hypergeometric series in the three distinguished cases \eqref{eq1.CBD}. In section \ref{sect8}, we deduce from these results Theorem \ref{thmD}. As a  corollary, we obtain symplectic and orthogonal versions of the discrete B-spline.  The last section \ref{sect9} contains a series of remarks. At the very end is a list of symbols.

\section{Multiparameter and dual Schur functions}\label{sect2}

Here we state a few results from Olshanski \cite[section 4]{Ols-Cauchy} which will be used in the sequel. (As pointed out in \cite{Ols-Cauchy}, these results can also be extracted from the earlier paper Molev \cite{Molev}.)

\begin{definition}[Multiparameter Schur polynomials]\label{def2.A}
Let $(c_0,c_1,c_2,\dots)$ be an infinite sequence of parameters and consider the monic polynomials
\begin{equation*}
(x\mid c_0,c_1,\dots)^m:=(x-c_0)\dots(x-c_{m-1}), \qquad m=0,1,2,\dots\,.
\end{equation*}
 The $N$-variate \emph{multiparameter Schur polynomials} are defined by 
 \begin{equation*}
S_{\mu, N}(x_1,\dots,x_N\mid c_0,c_1,\dots)
:=\frac{\det[(x_i\mid c_0,c_1,\dots)^{\mu_r+N-r}]_{i,r=1}^N}{V(x_1,\dots,x_N)}, \qquad \mu\in\Sign^+_N.
\end{equation*}
This is a particular case of generalized Schur polynomials (see section \ref{sect1.5}). If $c_0=c_1=\dots=0$, they turn into the conventional Schur polynomials. 
\end{definition}

\begin{definition}[Dual Schur functions]\label{def2.B}
We apply the definition of Schur-type functions (section \ref{sect1.5}) by taking
\begin{equation*}
\phi_m(t)=\frac1{(y\mid c_1,c_2,\dots)^m}
\end{equation*}
(note a shift by $1$ in the indexation of the parameters). The corresponding $N$-variate functions are denoted by $\si_{\mu, N}(y_1,\dots,y_N\mid c_1,c_2,\dots)$:
\begin{equation}\label{eq2.B}
\si_{\mu, N}(y_1,\dots,y_N\mid c_1,c_2,\dots)
:=\frac{\det\left[\dfrac1{(y_j\mid c_1,c_2,\dots)^{\mu_r+N-r}}\right]_{j,r=1}^N}
{\det\left[\dfrac1{(y_j\mid c_1,c_2,\dots)^{N-r}}\right]_{j,r=1}^N}.
\end{equation}
Following Molev \cite{Molev} we call them the ($N$-variate) \emph{dual Schur functions}. If $c_1=c_2=\dots=0$, they turn into the conventional Schur polynomials in the variables $y_1^{-1},\dots,y_N^{-1}$. 
\end{definition}

\begin{lemma}\label{lemma1}
The dual Schur functions  \eqref{eq2.B} possess the following stability property
\begin{equation}\label{eq.10}
\si_{\mu, N}(y_1,\dots,y_N\mid c_1,c_2,\dots)\big|_{y_N=\infty}
=\begin{cases} \si_{\mu, N-1}(y_1,\dots,y_{N-1}\mid c_2,c_3,\dots), & \ell(\mu)\le N-1,\\
0, & \ell(\mu)=N. \end{cases}
\end{equation}
\end{lemma}

\begin{proof}
See \cite[Lemma 4.5]{Ols-Cauchy}.
\end{proof}

\begin{lemma}\label{lemma2}
One has
\begin{equation}\label{eq.8}
\det\left[\dfrac1{(y_j\mid c_1,c_2,\dots)^{N-r}}\right]_{j,r=1}^N  =
(-1)^{N(N-1)/2}\,\frac{V(y_1,\dots,y_N)}{\prod_{j=1}^N(y_j-c_1)\dots(y_j-c_{N-1})}.
\end{equation}
\end{lemma}

\begin{proof}
See \cite[Lemma 4.6]{Ols-Cauchy}.
\end{proof}

\begin{lemma}\label{lemma3}
The dual Schur functions in $N$ variables form a topological basis in the subalgebra
of $\C[[y_1^{-1},\dots,y_N^{-1}]]$ formed by the symmetric power series.
\end{lemma}

\begin{proof}
See \cite[Lemma 4.7]{Ols-Cauchy}.
\end{proof}

\begin{proposition}[Cauchy-type identity]\label{prop2.A}
For $K\le N$ one has
\begin{multline}\label{eq.16}
\sum_{\mu\in\Sign^+_K}S_{\mu, N}(x_1,\dots,x_N\mid c_0,c_1,\dots)
\si_{\mu, K}(y_1,\dots,y_K\mid c_{N-K+1},c_{N-K+2},\dots)
\\
=\prod_{j=1}^K\frac{(y_j-c_0)\dots(y_j-c_{N-1})}{(y_j-x_1)\dots(y_j-x_N)},
\end{multline}
where both sides are regarded as elements of the algebra of formal series in $y_1^{-1},\dots, y_k^{-1}$. 
\end{proposition}

\begin{proof}
See \cite[Proposition 4.8]{Ols-Cauchy}.
\end{proof}

\section{Coherency property for special multiparameter Schur polynomials}\label{sect3}

In the next proposition we use the normalized Jacobi polynomials $\wt P^{(a,b)}_{\nu,N}$ defined in \eqref{eq1.K}, the mutiparameter Schur polynomials (Definition \ref{def2.A}) corresponding to the special sequence of parameters $(\epsi^2,(\epsi+1)^2, (\epsi+2)^2,\dots)$, and the conventional Schur polynomials $S_{\mu,N}$. Recall that $\epsi=(a+b+1)/2$. 

Given $\nu\in\Sign^+_N$, we set
\begin{equation}\label{eq3.B}
n_i:=\nu_i+N-i, \qquad 1\le i\le N.
\end{equation}

\begin{proposition}[Binomial formula for Jacobi polynomials]\label{prop3.A}
Let $\nu\in\Sign^+_N$. One has
\begin{multline}\label{eq3.C}
\wt P^{(a,b)}_{\nu, N}(1+\al_1,\dots,1+\al_N)\\
=\sum_{\mu\in\Sign^+_N}\frac{S_{\mu, N}((n_1+\epsi)^2,\dots,(n_N+\epsi)^2\mid \epsi^2, (\epsi+1)^2,
\dots)}{C(N,\mu; a)}S_{\mu, N}(\al_1,\dots,\al_N),
\end{multline}
where 
\begin{equation}\label{eq3.D}
C(N,\mu;a)= 2^{|\mu|}\, \prod_{i=1}^N
\frac{\Gamma(\mu_i+N-i+1)\Gamma(\mu_i+N-i+a+1)} {\Gamma(N-i+1)\Gamma(N-i+a+1)}
\end{equation}
\end{proposition}

\begin{proof}
See Okounkov--Olshanski \cite[Theorem 1.2]{OO-KirSem} and \cite[Proposition 7.4]{OO-Jacobi}.
\end{proof}

Let $\nu\in\Sign^+_N$ and $\ka\in\Sign^+_K$, where $N>K\ge1$. Recall that the quantities $\La^N_K(\nu,\ka)$ are the coefficients in the expansion 
\begin{equation}\label{eq3.A}
\wt P^{(a,b)}_{\nu, N}(x_1,\dots,x_K,1,\dots,1)
=\sum_{\ka\in\Sign^+_K}\La^N_K(\nu,\ka)
\wt P^{(a,b)}_{\ka, K}(x_1,\dots,x_K).
\end{equation}
Next, let $\mu\in\Sign^+_K$ and regard it also as an element of $\Sign^+_N$ by adjusting $N-K$ zeroes.  By analogy with \eqref{eq3.B} we also set 
$$
k_i:=\ka_i+K-i, \quad 1\le i\le K.
$$ 

\begin{theorem}[The coherency property]\label{thm3.A}
In this notation, the following relation holds
\begin{multline}\label{eq3.F}
\frac{S_{\mu, N}((n_1+\epsi)^2,\dots,(n_N+\epsi)^2\mid\epsi^2, (\epsi+1)^2,
\dots)}{C(N,\mu;a)}\\
=\sum_{\ka\in\Sign^+_K}
\La^N_K(\nu,\ka)\frac{S_{\mu, K}((k_1+\epsi)^2,\dots,(k_K+\epsi)^2\mid\epsi^2,
(\epsi+1)^2,\dots)}{C(K,\mu;a)}.
\end{multline}
\end{theorem}

\noindent\emph{Comments}. 1. The sum on the right-hand side is finite, because for any $\nu$,
there are only finitely many $\ka$'s for which $\La^N_K(\nu,\ka)\ne0$. Indeed, a necessary condition for $\La^N_K(\nu,\ka)\ne0$ is $\ka_1\le\nu_1$, as it is seen from the branching rule for multivariate Jacobi polynomials (\cite[Proposition 7.5]{OO-Jacobi}).

2. A similar relation holds in the case of type $\Aa$, see \cite[(5.6)]{BO-AdvMath} and \cite[(10.30)]{OO-AA}. 

3. Let $\nu\in\Sign^+_N$ be fixed and $\mu$ range over $\Sign^+_K$. Then $\La^N_K(\nu,\ccdot)$ is a unique finitely supported solution to the system of linear equations produced by the coherency relations \eqref{eq3.F}. This follows from the fact that the multiparameter Schur polynomials on the right-hand side of \eqref{eq3.F} form a basis of the algebra of symmetric $K$-variate polynomials, and this algebra separates the $K$-point configurations of the form
$$
(x_1,\dots,x_k)=((k_1+\epsi)^2,\dots,(k_K+\epsi)^2)
$$
corresponding to the signatures $\ka\in\Sign^+_K$.  

\begin{proof}
We apply the binomial formula \eqref{eq3.C} and the definition \eqref{eq3.A}. Make the change
$N\to K$ and $\nu\to \ka$; then the equation \eqref{eq3.C} turns into
\begin{multline*}
\wt P^{(a,b)}_{\ka, K}(1+\al_1,\dots,1+\al_K)\\
=\sum_{\mu\in\Sign^+_K}\frac{S_{\mu, K}((k_1+\epsi)^2,\dots,(k_K+\epsi)^2|\epsi^2,(\epsi+1)^2,\dots)}
{C(K,\mu;a)}S_{\mu, K}(\al_1,\dots,\al_K).
\end{multline*}

Substituting this into \eqref{eq3.A} and interchanging summation gives

\begin{multline}\label{eq.3}
\wt P^{(a,b)}_{\nu, N}(1+\al_1,\dots,1+\al_K,1,\dots,1) \\
=\sum_{\mu\in\Sign^+_K}\left(\sum_{\ka\in\Sign^+_K}
\La^N_K(\nu,\ka)\frac{S_{\mu, K}((k_1+\epsi)^2,\dots,(k_K+\epsi)^2|\epsi^2,
(\epsi+1)^2,\dots)}{C(K,\mu;a)}\right)\\
\times S_{\mu, K}(\al_1,\dots,\al_K).
\end{multline}

On the other hand, specializing $\al_{K+1}=\dots=\al_N=0$ in the binomial
formula \eqref{eq3.C} gives

\begin{multline}\label{eq.4}
\wt P^{(a,b)}_{\nu, N}(1+\al_1,\dots,1+\al_K,
1,\dots,1)\\
=\sum_{\mu\in\L_K}\frac{S_{\mu, N}((n_1+\epsi)^2,\dots,(n_N+\epsi)^2|\epsi^2,(\epsi+1)^2,\dots)}
{C(N,\mu;a)}S_{\mu, K}(\al_1,\dots,\al_K).
\end{multline}

Comparing \eqref{eq.3} with \eqref{eq.4} and equating the coefficients of the
Schur polynomials $S_{\mu, K}(\al_1,\dots,\al_K)$ we are led to the desired
formula \eqref{eq3.F}.
\end{proof}

\section{Cauchy-type identity and determinantal formula for the matrix entries $\La^N_K(\nu,\ka)$: proof of Theorems \ref{thmA} and \ref{thmB}}\label{sect4}

In this section we fix the Jacobi parameters $(a,b)$ such that $a>-1$, $b>-1$, and $a+b\ge-1$, so that the parameter $\epsi:=\frac12(a+b+1)$ is nonnegative. We are dealing with  the functions $g_k(t)$, the Schur-type 
functions $G_\ka(t_1,\dots,t_K)$, the characteristic function $F_N(t;\nu;\epsi)$ of a signature $\nu\in\Sign^+_N$, and the space $\FF(\epsi,L)$ related to the grid $\AAA(\epsi,L)$ consisting of the points $A_m:=L+\epsi+m-1$, where $m=1,2,\dots$\,. All these objects were defined in sections \ref{sect1.7}--\ref{sect1.8}. We assume that $N>K$ and $L=N-K+1$, so that $L\ge2$. 

\begin{lemma}\label{lemma4.A}
For any $\nu\in\Sign^+_N$, the function $\prod_{j=1}^KF_N(t_j;\nu;\epsi)$ can be written, in a unique way, as  a finite linear combination of the functions $G_\ka(t_1,\dots,t_K)$, where $\ka$ ranges over\/ $\Sign^+_K$. 
\end{lemma}

\begin{proof}
\emph{Step} 1. From the definition \eqref{eq1.g_k} of the functions $g_k(t)$ it is seen that they are even, rational, and regular at $t=\infty$. The function $g_0(t)$ is the constant $1$. If $k\ge1$, then the singularities of $g_k(t)$ are simple poles contained in the set $\{\pm A_1,\dots, \pm A_k\}$ and, moreover, the residue at $\pm A_k$ is nonzero. It follows that the functions $g_k(t)$ form a basis of $\FF(\epsi,L)$.

\emph{Step} 2. 
The claim of the lemma is obviously equivalent to the following: there exists a unique finite expansion of the form 
\begin{equation}\label{eq4.exp}
\det[g_{K-i}(t_j)]_{i,j=1}^K\prod_{j=1}^KF_N(t_j;\nu;\epsi)
=\sum_{k_1>\dots>k_K\ge0}(\cdots)\det[g_{k_i}(t_j)]_{i,j=1}^K,
\end{equation}
where the dots stand for some coefficients. 

Write the left-hand side as
$$
\det[g_{K-i}(t_j)F_N(t_j;\nu;\epsi)]_{i,j=1}^K.
$$
By virtue of step 1, the existence and uniqueness of the expansion \eqref{eq4.exp} is  reduced to the following claim concerning functions of a single variable $t$: for each $m=0,\dots, K-1$, the function 
$ g_m(t) F_N(t;\nu;\epsi)$ lies in the space $\FF(\epsi,L)$.

\emph{Step} 3. 
Let us prove the latter claim. It is clear that $ g_m(t) F_N(t;\nu;\epsi)$ is even, rational, and regular at infinity. It remains to examine its singularities. From the definition \eqref{eq1.F_N} of $F_N(t;\nu;\epsi)$ it follows that its singularities are simple poles contained in the set 
$$
\{\pm (N+\epsi), \pm(N+\epsi+1), \pm(N+\epsi+2), \dots\},
$$
while the singularities of $g_m(t)$ with $m\ne0$ are simple poles contained in the set 
$$
\{\pm (L+\epsi), \pm(L+\epsi+1),\dots, \pm(L+\epsi+m-1)\}.
$$
Since $m\le K-1$ and $L=N-K+1$, the both sets are disjoint. Furthermore, they are contained in $-(\AAA(\epsi,L)\cup\AAA(\epsi,L)$. Thus, the product $ g_m(t) F_N(t;\nu;\epsi)$ has only simple poles, all of which are contained in $-(\AAA(\epsi,L)\cup\AAA(\epsi,L)$. This proves that $ g_m(t) F_N(t;\nu;\epsi)$ lies in the space $\FF(\epsi,L)$.  
\end{proof}

For $\ka\in\Sign^+_K$ we set
\begin{equation}\label{eq4.dkappa}
d_K(\ka;\epsi):=\prod_{1\le i<j\le K}\frac{(k_i+\epsi)^2-(k_j+\epsi)^2}{(k^0_i+\epsi)^2-(k^0_j+\epsi)^2}=\prod_{1\le i<j\le K}\frac{(\ka_i+K-i+\epsi)^2-(\ka_j+K-j+\epsi)^2}{(K-i+\epsi)^2-(K-j+\epsi)^2}.
\end{equation}

\begin{theorem}\label{thm4.A}
The precise form of the expansion in Lemma \ref{lemma4.A} is 
\begin{equation}\label{MasterEq}
\prod_{j=1}^KF_N(t_j;\nu;\epsi)=\sum_{\ka\in\Sign^+_K}\frac{\La^N_K(\nu,\ka)}{d_K(\ka;\epsi)}\,G_\ka(t_1,\dots,t_K).
\end{equation}
\end{theorem}

This is Theorem \ref{thmA} from section \ref{sect1.7}

Before proceeding to the proof we need a preparation.  In the next lemma we are dealing with a particular case of the multiparameter Schur polynomials (Definition \ref{def2.A}) and dual Schur functions (Definition \ref{def2.B}). We assume $\mu,\ka\in\Sign^+_K$ and write $\mu\subseteq\ka$ if $\mu_i\le\ka_i$ for all $i=1,\dots,K)$. 

\begin{lemma}\label{lemma4.B}
For  $\ka\in\Sign^+_K$, one has
\begin{equation}\label{eq4.B}
\frac{G_\ka(t_1,\dots,t_K; a,\epsi,L)}{d_K(\ka;\epsi)}=\sum_{\mu:\, \mu\subseteq\ka} A_{\mu,\ka}\, \si_{\mu, K}(t_1^2,\dots,t_K^2\mid (L+\epsi)^2, (L+\epsi+1)^2,\dots),
\end{equation}
where the coefficients $A_{\mu,\ka}$ are defined by
\begin{align}
A_{\mu,\ka}: &=\prod_{i=1}^K\frac{(L)_{m_i}(L+a)_{m_i}(a+1)_{K-i}(K-i)!}{(a+1)_{m_i}m_i! (L)_{K-i}(L+a)_{K-i}}\label{4.product}\\
& \times S_{\mu, K}((k_1+\epsi)^2,\dots,(k_K+\epsi)^2\mid \epsi^2, (\epsi+1)^2,\dots)
\end{align}
with $m_i:=\mu_i+K-i$.
\end{lemma}

Due to the constraint $\mu\subseteq\ka$, the expansion \eqref{eq4.B} is finite. Here is an immediate corollary of the lemma. 

\begin{corollary}
The rational function on the left-hand side of \eqref{eq4.B}, viewed as a function of the variables $t_1^{-1},\dots,t_K^{-1}$, is regular about the point $(0,\dots,0)$ and its value at this point equals $1$. 
\end{corollary} 

Indeed, this follows from the fact that the coefficient $A_{\emptyset,\ka}$ corresponding to the signature $\emptyset=(0,\dots0)$ equals $1$.

\begin{proof}[Proof of Lemma \ref{lemma4.B}]
\emph{Step} 1. We rewrite the definition \eqref{eq1.g_k} of the function $g_k(t)$ in the form
\begin{equation*}
g_k(t)=\sum_{m=0}^\infty X(k,m)Y(t,m),
\end{equation*}
where
\begin{gather}
X(k,m):=\frac{(-k)_m(k+2\epsi)_m(L)_m(L+a)_m}{(a+1)_m m!}, \label{4.3.1} \\
Y(t,m):=\frac1{(-t+\epsi+L)_m(t+\epsi+L)_m}.\label{4.3.9}
\end{gather}
The idea is to separate the terms depending on $t$ from those depending on $k$. The formally infinite series in fact terminates due to the factor $(-k)_m$. 

From this presentation it follows that for $\ka\in\Sign^+_K$,
\begin{equation}\label{4.3.2}
\det[g_{k_i}(t_j)]_{i,j=1}^K=\sum_{m_1>\dots>m_r\ge0}\det[X(k_i,m_r)]_{i,r=1}^K\det[Y(t_j,m_r)]_{j,r=1}^K.
\end{equation}

Let $\mu\in\Sign^+_K$ be the signature corresponding to the tuple $(m_1,\dots,m_K)$, meaning that $\mu_i=m_i-(K-i)$ for $i=1,\dots,K$. Observe that $\det[X(k_i,m_r)]_{i,r=1}^K=0$ unless $m_i\le k_i$ for all $i=1,\dots,K$. Indeed, suppose the opposite; then there exists an index $s$ such that $m_s>k_s$. It follows that $m_i>k_r$ whenever $i\le s\le r$. Due to the factor $(-k)_m$ in \eqref{4.3.1}, for any such pair $(i,r)$, the corresponding entry $X(k_i,m_r)$ vanishes. But this in turns implies that the determinant vanishes. 

We have proved that the summation in \eqref{4.3.2} in fact goes over the signatures $\mu\subseteq\ka$. 

\emph{Step} 2. In particular, for $\ka=\emptyset$ the sum \eqref{4.3.2} reduces to a single summand:
\begin{equation}\label{4.3.3}
\det[g_{K-i}(t_j)]_{i,j=1}^K=\det[X(K-i,K-r)]_{i,r=1}^K\det[Y(t_j,K-r)]_{j,r=1}^K.
\end{equation}
From \eqref{4.3.2}, \eqref{4.3.3},  and the definition \eqref{eq1.Gkappa} of the function $G_\ka(t_1,\dots,t_K)$ we obtain
\begin{equation}\label{4.3.4}
G_\ka(t_1,\dots,t_K)=\sum_{\mu:\, \mu\subseteq\ka}\frac{\det[X(k_i,m_r)]_{i,r=1}^K}{\det[X(K-i,K-r)]_{i,r=1}^K}\frac{\det[Y(t_j,m_r)]_{j,r=1}^K}{\det[Y(t_j,K-r)]_{j,r=1}^K}.
\end{equation}

We are going to show that
\begin{equation}\label{4.3.5}
\frac1{d_K(\ka;\epsi)}\,\frac{\det[X(k_i,m_r)]_{i,r=1}^K}{\det[X(K-i,K-r)]_{i,r=1}^K}=(-1)^{|\mu|}A_{\mu,\ka}
\end{equation}
and 
\begin{equation}\label{4.3.6}
\frac{\det[Y(t_j,m_r)]_{j,r=1}^K}{\det[Y(t_j,K-r)]_{j,r=1}^K}=(-1)^{|\mu|}\si_{\mu, K}(t_1^2,\dots,t_K^2\mid (L+\epsi)^2, (L+\epsi+1)^2,\dots).
\end{equation}
Then the three formulas  \eqref{4.3.4}, \eqref{4.3.5}, and \eqref{4.3.6} will imply the lemma.  

\emph{Step} 3. Let us prove \eqref{4.3.5}. From the definition of $X(k,m)$ (see \eqref{4.3.1}) we have
\begin{multline}\label{4.3.7}
\frac{\det[X(k_i,m_r)]_{i,r=1}^K}{\det[X(K-i,K-r)]_{i,r=1}^K}=(\text{the product \eqref{4.product}})\\
\times\frac{\det[(-k_i)_{m_r}(k_i+2\epsi)_{m_r}]}{\det[(-(K-i))_{K-r}(K-i+2\epsi)_{K-r}]}.
\end{multline}
Observe that
\begin{multline}\label{4.3.8}
(-k)_m(k+2\epsi)_m=\prod_{\ell=0}^{m-1}(-k+\ell)(k+2\epsi+\ell)=(-1)^m\prod_{l=0}^{m-1}((k+\epsi)^2-(\epsi+\ell)^2)\\
=(-1)^m((k+\epsi)^2\mid \epsi^2, (\epsi+1)^2,\dots)^m.
\end{multline}
It follows that 
\begin{equation}
\frac{\det[(-k_i)_{m_r}(k_i+2\epsi)_{m_r}]}{\det[(-(K-i))_{K-r}(K-i)_{K-r}]}=(-1)^{|\mu|}\frac{\det[((k_i+\epsi)^2\mid \epsi^2, (\epsi+1)^2,\dots)^{m_r}]}{\det[((K-i+\epsi)^2\mid \epsi^2, (\epsi+1)^2,\dots)^{K-r}}.
\end{equation}
Since
$$
\det[((K-i+\epsi)^2\mid \epsi^2, (\epsi+1)^2,\dots)^{K-r}]=\prod_{1\le i<j\le K}((K-i)^2-(K-j)^2),
$$
we may rewrite  \eqref{4.3.7} as
\begin{multline*}
\frac{\det[X(k_i,m_r)]_{i,r=1}^K}{\det[X(K-i,K-r)]_{i,r=1}^K}=(\text{the product \eqref{4.product}})\\
\times
(-1)^{|\mu|}d_K(\ka;\epsi)\, \frac{\det[((k_i+\epsi)^2\mid \epsi^2, (\epsi+1)^2,\dots)^{m_r}]}{\prod\limits_{1\le i<j\le K}((k_i+\epsi)^2-(k_j+\epsi)^2)}.
\end{multline*}
Substitute this into \eqref{4.3.5}; then the product over $i<j$ is cancelled. Using the definition of dual Schur functions (see \eqref{eq2.B}), we conclude that  the resulting expression is equal to $(-1)^{|\mu|} A_{\mu,\ka}$, as desired. 

\emph{Step 4}. Let us prove \eqref{4.3.6}. Similarly to \eqref{4.3.8} we have
\begin{equation}
(-t+\epsi+L)_m(t+\epsi+L)_m=(-1)^m (t^2\mid (L+\epsi)^2,(L+\epsi+1)^2,\dots)^m.
\end{equation}
From this and the definition of $Y(t,m)$ (see \eqref{4.3.9}) we obtain
\begin{equation}
\frac{\det[Y(t_j,m_r)]_{j,r=1}^K}{\det[Y(t_j,K-r)]_{j,r=1}^K}
=(-1)^{|\mu|}\frac{\det\left[\dfrac1{(t_j^2\mid (\epsi+L)^2,(\epsi+L+1)^2,\dots)^{m_r}}\right]}{\det\left[\dfrac1{(t_j^2\mid (\epsi+L)^2,(\epsi+L+1)^2,\dots)^{K-r}}\right]}.
\end{equation}
By the  definition of the dual Schur functions this equals the right-hand side of \eqref{4.3.6}.

This completes the proof of the lemma. 
\end{proof}

\begin{proof}[Proof of Theorem \ref{thm4.A}]
\emph{Step} 1. 
We begin with the coherency relation \eqref{eq3.F}, which we write in the form
\begin{multline*}
S_{\mu, N}((n_1+\epsi)^2,\dots,(n_N+\epsi)^2|\epsi^2, (\epsi+1)^2,
\dots)\\
=\sum_{\ka\in\Sign^+_K}
\La^N_K(\nu,\ka)S_{\mu, K}((k_1+\epsi)^2,\dots,(k_K+\epsi)^2|\epsi^2,
(\epsi+1)^2,\dots)\frac{C(N,\mu;a)}{C(K,\mu;a)}.
\end{multline*}
Here $\mu\in\Sign^+_K$ is arbitrary; recall also that the sum is finite, because for each fixed $\nu$, the quantity $\La^N_K(\nu,\ka)$ is nonzero only for finitely many $\ka$'s. 

Let us multiply the both sides by 
$$
\si_{\mu, K}(t^2_1,\dots,t^2_K\mid (\epsi+N-K+1)^2,\,(\epsi+N-K+2)^2,\,\dots)
$$ 
and sum over all $\mu\in\Sign^+_K$, which makes sense in the algebra of formal power series in $t^{-2}_1,\dots,t^{-2}_K$ due to Lemma \ref{lemma3}. The resulting equality has the form
\begin{multline}\label{4.3.10}
\sum_{\mu\in\Sign^+_K}S_{\mu, N}((n_1+\epsi)^2,\dots,(n_N+\epsi)^2|\epsi^2,(\epsi+1)^2,\dots)\\
\times \si_{\mu, K}(t^2_1,\dots,t^2_K\mid (\epsi+N-K+1)^2,\,(\epsi+N-K+2)^2,\,\dots)\\
=\sum_{\mu\in\Sign^+_K}\sum_{\ka\in\Sign^+_K}
\La^N_K(\nu,\ka)S_{\mu, K}((k_1+\epsi)^2,\dots,(k_K+\epsi)^2|\epsi^2,
(\epsi+1)^2,\dots)\frac{C(N,\mu;a)}{C(K,\mu;a)}\\
\times \si_{\mu, K}(t^2_1,\dots,t^2_K\mid (\epsi+N-K+1)^2,\,(\epsi+N-K+2)^2,\,\dots).
\end{multline}

We will show that this equation can be reduced to \eqref{MasterEq}.

\emph{Step} 2. Examine the left-hand side of \eqref{4.3.10}. We apply to it the Cauchy-type identity (see \eqref{eq.16})
\begin{multline*}
\sum_{\mu:\,\ell(\mu)\le K}S_{\mu, N}(x_1,\dots,x_N\mid c_0,c_1,\dots)
\si_{\mu \mid K}(y_1,\dots,y_K\mid c_{N-K+1},c_{N-K+2},\dots)
\\
=\prod_{j=1}^K\frac{(y_j-c_0)\dots(y_j-c_{N-1})}{(y_j-x_1)\dots(y_j-x_N)},
\end{multline*}
where we specialize 
\begin{equation*}
x_1:=(n_1+\epsi)^2,\;\dots,\; x_N:=(n_N+\epsi)^2, \qquad y_1:=t_1^2,\; \dots,\;y_K:=t_K^2
\end{equation*}
and
$$
c_i:=(\epsi+i)^2, \quad i=0,1,\dots\,.
$$
Then the result gives us the left-hand side of \eqref{MasterEq}. 

\emph{Step} 3. 
We proceed now to the right-hand side of \eqref{4.3.10}. Here we may interchange the two summations, because $\ka$ actually ranges over a finite set depending only on $\nu$. Then we obtain the double sum
$$
\sum_{\ka\in\Sign^+_K}\La^N_K(\nu,\ka) \sum_{\mu\in\Sign^+_K}(\cdots),
$$
where the interior sum over $\mu$ has the form
\begin{multline}\label{4.3.11}
\sum_{\mu\in\Sign^+_K}
S_{\mu, K}((k_1+\epsi)^2,\dots,(k_K+\epsi)^2|\epsi^2,
(\epsi+1)^2,\dots)\frac{C(N,\mu;a)}{C(K,\mu;a)}\\
\times \si_{\mu, K}(t^2_1,\dots,t^2_K\mid (\epsi+N-K+1)^2,\,(\epsi+N-K+2)^2,\,\dots).
\end{multline}
We will prove that this sum equals 
$$
\frac{G_\ka(t_1,\dots,t_K)}{d_K(\ka,\epsi)},
$$
which in turn will imply that that the right-hand side of \eqref{4.3.10} coincides with the right-hand side of \eqref{MasterEq}.

Comparing \eqref{4.3.11} with the result of Lemma \ref{lemma4.B} we see that it remains to check the equality
\begin{equation}\label{4.3.14}
\frac{C(N,\mu;a)}{C(K,\mu;a)}=\prod_{i=1}^K\frac{(L)_{m_i}(L+a)_{m_i}(a+1)_{K-i}(K-i)!}{(a+1)_{m_i}m_i! (L)_{K-i}(L+a)_{K-i}}.
\end{equation}

The quantities on the left-hand side were defined in \eqref{eq3.D}; we have
\begin{equation}\label{4.3.12}
C(N,\mu;a)= 2^{|\mu|}\, \prod_{i=1}^N
\frac{\Gamma(\mu_i+N-i+1)\Gamma(\mu_i+N-i+a+1)} {\Gamma(N-i+1)\Gamma(N-i+a+1)}
\end{equation}
and
\begin{equation}\label{4.3.13}
C(K,\mu;a)= 2^{|\mu|}\, \prod_{i=1}^K
\frac{\Gamma(\mu_i+K-i+1)\Gamma(\mu_i+K-i+a+1)} {\Gamma(K-i+1)\Gamma(K-i+a+1)}
\end{equation}
Since $\ell(\mu)\le K$, the product in \eqref{4.3.12} can in fact be restricted to $i=1,\dots,K$. After that the equality \eqref{4.3.14} is readily checked.  

This completes the proof of the theorem.
\end{proof}

\begin{remark}
The main ingredients of the proof of Theorem \ref{thm4.A} are two formulas involving  Schur-type functions: the Cauchy-type identity \eqref{eq.16} and the coherency relation \eqref{eq3.F}. A similar mechanism works in the case of unitary groups \cite{BO-AdvMath}. 
\end{remark}

\begin{remark}
 In the statement of Theorem \ref{thm4.A} it was assumed $N>K$. In this remark we examine what happens for $N=K$. Looking at the proof one sees that it works for $N=K$, with the understanding that $\La^K_K(\nu,\ka)=\de_{\nu,\ka}$. Then the result reduces to
\begin{equation*}
\prod_{j=1}^KF_K(t_j;\nu;\epsi)=\sum_{\ka\in \Sign^+_K}\de_{\nu,\ka}\frac{G_\ka(t_1,\dots,t_K;a,\epsi,1)}{d_K(\ka;\epsi)},
\end{equation*}
where we used a more detailed notation $G_\ka(t_1,\dots,t_K;a,\epsi,L)$ instead of $G_\ka(t_1,\dots,t_K)$ and then specialized $L$ to $1$, because $N=K$ means $L=1$. 

Rewrite this equality as
\begin{equation*}
G_\ka(t_1,\dots,t_K;a,\epsi,1)=d_K(\ka;\epsi)\prod_{j=1}^KF_K(t_j;\ka;\epsi)
\end{equation*}
or else, in a more detailed form (we use the definition \eqref{eq1.F_N}),
\begin{equation}\label{4.7.1}
G_\ka(t_1,\dots,t_K;a,\epsi,1)=\prod_{j=1}^K\prod_{i=1}^K\frac{t_j^2-(k^0_i+\epsi)^2}{t_j^2-(\ka_i+\epsi)^2}\cdot \prod_{1\le i<j\le K}\frac{(k_i+\epsi)^2-(k_j+\epsi)^2}{(k_i^0+\epsi)^2-(k_j^0+\epsi)^2},
\end{equation}
where $k_i:=\ka_i+K-i$ and $k^0_i:=K-i$. 

Formula \eqref{4.7.1} can be checked directly, as follows.   

For $L=1$, the definition \eqref{eq1.g_k} drastically simplifies and takes the form
\begin{equation}\label{4.7.2}
g_k(t;a,\epsi,1):={}_3F_2\left[\begin{matrix}-k,\, k+2\epsi,\, 1\\-t+1+\epsi,\,
t+1+\epsi \end{matrix}\Biggl|1\right]=\frac{t^2-\epsi^2}{t^2-(k+\epsi)^2},
\end{equation}
where the second equality follows from a well-known summation formula due to Saalsch\"utz (Bailey \cite[section 2.2, (1)]{Bailey}). The resulting expression \eqref{4.7.2} is precisely the special case of \eqref{4.7.1} corresponding to $K=1$. 

Next, using \eqref{4.7.2}, we obtain for $K\ge2$
$$
\det[g_{k_i}(t_j;a,\epsi,1)]_{i,j=1}^K=\prod_{j=1}^K(t_j^2-\epsi^2)\cdot\det\left[\dfrac1{t_j^2-(k_i+\epsi)^2}\right]_{i,j=1}^K.
$$
The determinant on the right is a Cauchy determinant. It follows that
$$
\det[g_{k_i}(t_j;a,\epsi,1)]_{i,j=1}^K=(\cdots)\prod_{i,j=1}^K\frac1{t_j^2-(k_i+\epsi)^2}\cdot\prod_{1\le i<j\le K}((k_i+\epsi)^2-(k_j+\epsi)^2),
$$
where the dots denote an expression which depends only on $t_j$'s but not on  $k_i$'s. Dividing this by the similar expression with $k_i=k_i^0$ we finally obtain \eqref{4.7.1}.
\end{remark}

Recall that  the functions $g_k(t)$ constitute a basis of $\FF(\epsi,L)$ (step 1 of the proof of Lemma \ref{lemma4.A}).  
Given a function $f\in \FF(\epsi,L)$, we denote by $(f: g_k)$ the $k$th coefficient ($k=0,1,2,\dots$) in the expansion of $f$ in the basis $\{g_k\}$.

The next result is a corollary of Theorem \ref{thm4.A}. 

\begin{theorem}\label{thm4.B}
The following determinantal formula holds
\begin{equation}\label{4.JacobiTrudi}
\frac{\La^N_K(\nu,\ka)}{d_K(\ka;\epsi)}
=\det\left[(g_{K-j}F_N:\,  g_{k_i})\right]_{i,j=1}^K.
\end{equation} 
\end{theorem}

This is Theorem \ref{thmB} from section \ref{sect1.8}. 

\begin{proof}
Recall (formula \eqref{MasterEq}) that
\begin{equation}\label{MasterEq1}
\prod_{j=1}^KF_N(t_j;\nu;\epsi)=\sum_{\ka\in\Sign^+_K}\frac{\La^N_K(\nu,\ka)}{d_K(\ka;\epsi)}\,G_\ka(t_1,\dots,t_K)
\end{equation}
and (the definition \eqref{eq1.Gkappa}) that
\begin{equation}\label{eq1.Gkappa1}
G_{\ka,K}(t_1,\dots,t_K)=\frac{\det[g_{\ka_i+K-i}(t_j)]_{i,j=1}^K}{\det[g_{K-i}(t_j)]_{i,j=1}^K}, \quad \ka\in\Sign^+_K\,.
\end{equation}
Substituting \eqref{eq1.Gkappa1} into \eqref{MasterEq1} and multiplying both sides by $\det[g_{K-i}(t_j)]_{i,j=1}^K$ we obtain 
\begin{equation}\label{eq4.C}
\det[g_{K-i}(t_j)F_N(t_j;\nu;\epsi)]_{i,j=1}^K=\sum_{k_1>\dots>k_K\ge0} \frac{\La^N_K(\nu,\ka)}{d_K(\ka;\epsi)}\, \det[g_{k_i}(t_j)]_{i,j=1}^K.
\end{equation}

Next, recall that the functions $g_0(t), g_1(t),\dots$ form a basis of the space $\FF(\epsi,L)$ and the functions  $g_{K-i}(t)F_N(t;\nu;\epsi)$ lie in this space (see the proof of Lemma \ref{lemma4.A}, steps 1 and 3). 

Now let us abbreviate
$$
h_{K-i}(t):=g_{K-i}(t)F_N(t;\nu;\epsi).
$$
From the above it follows that there exists a \emph{unique} expansion
$$
\det[h_{K-i}(t_j)]_{i,j=1}^K=\sum_{k_1>\dots>k_K\ge0} c(k_1,\dots,k_K) \det[g_{K-i}(t_j)]_{i,j=1}^K, 
$$
valid for all $t_1,\dots,t_K$. Furthermore, the coefficients of this expansion are given by
$$
c(k_1,\dots,k_K)=\det\left[(h_{K-j}:\,  g_{k_i})\right]_{i,j=1}^K.
$$
It follows that \eqref{eq4.C} implies \eqref{4.JacobiTrudi}.
\end{proof}

A similar determinantal formula holds for the unitary groups, see \cite[Proposition 6.2]{BO-AdvMath}. Notice that \eqref{4.JacobiTrudi} resembles the classical Jacobi--Trudi formula for the Schur symmetric polynomials, and the above argument is similar to the derivation of the latter formula from the Cauchy identity.

\section{Expansion in the basis $\{g_k(t)\}$ in the general case: proof of Theorem \ref{thmC}}\label{sect5}

Recall that we are dealing with the functions defined by \eqref{eq1.g_k}:
\begin{equation*}
g_k(t;a,\epsi,L):={}_4F_3\left[\begin{matrix}-k,\, k+2\epsi,\, L,\, L+a\\-t+L+\epsi,\,
t+L+\epsi,\, a+1\end{matrix}\Biggl|1\right], \qquad k=0,1,2,\dots\,.
\end{equation*}
Here $L\ge2$ is a positive integer, $a>-1$ and $\epsi\ge0$ are real parameters. We keep these parameters fixed and abbreviate $g_k(t):=g_k(t;a,\epsi,L)$.

We keep to the notation introduced in sect. \ref{sect1.8} and \ref{sect1.9}. In particular, $\{e_m: m\in\Z_{\ge0}\}$ is the basis of $\FF(\epsi,L)$ defined in \eqref{eq1.e} and $\Res_{t=A_m}(\phi(t))$ denotes the residue of $\phi(t)$ at the point $t=A_m$.

The Jacobi-Trudi-type formula given by Theorem \ref{thm4.B} reduces the computation of the matrix entries  $\La^N_K(\nu,\ka)$ to the following  one-dimensional problem (it was already stated in section \ref{sect1.9}): 

\begin{problem}
Given a function $\phi\in\FF(\epsi,L)$, how to compute the coefficients $(\phi:g_k)$ of its expansion in the basis $\{g_k\}$? Specifically, we need this for the functions $\phi=g_{K-j}F_N$. 
\end{problem}

In the present section we study the problem in the case of general Jacobi parameters (as before, the only constraints are those of \eqref{eq1.Jacobi}).

\begin{proposition}\label{prop5.A}
For any $\phi\in \FF(\epsi,L)$, one has  
\begin{equation}\label{eq5.coeff}
(\phi:g_k)=\begin{cases} \sum_{m\ge k}\Res_{t=A_m}(\phi(t))(e_m:g_k), & k\ge1, \\
\phi(\infty)+\sum_{m\ge1}\Res_{t=A_m}(\phi(t))(e_m: g_0), & k=0,
\end{cases}
\end{equation}
\end{proposition}

\begin{proof}
Write the expansion of $\phi$ in the basis $\{e_m\}$ as
$$
\phi=\sum_{m\ge0}(\phi: e_m)e_m.
$$
From this we obtain
\begin{equation}\label{eq5.coeff1}
(\phi: g_k)=\sum_{m\ge0}(\phi: e_m)(e_m:g_k), \quad k\in\Z_{\ge0}
\end{equation}

On the other hand, from the definition of the functions $e_m(t)$ it follows that  
\begin{equation}\label{eq5.coeff2}
(\phi: e_0)=\phi(\infty); \qquad (\phi: e_m)=\Res_{t=A_m}(\phi(t)), \quad m\ge1.
\end{equation}

Recall that $g_0=1$ and the only poles of the function $g_k$ with index $k\ge1$ are the points $A_{\pm \ell}$ with $1\le \ell\le k$. Therefore, for each $k\ge0$, the coefficients $(g_k:e_\ell)$ vanish unless $\ell\le k$. This triangularity property in turn implies that the coefficients $(e_m,g_k)$ vanish unless $m\ge k$. Thus, we may rewrite \eqref{eq5.coeff1} as
\begin{equation*}
(\phi: g_k)=\sum_{m\ge k}(\phi: e_m)(e_m:g_k), \quad k\in\Z_{\ge0}. 
\end{equation*}
Together with \eqref{eq5.coeff2}, this yields \eqref{eq5.coeff}. 
\end{proof}

To apply Proposition \ref{prop5.A}, we need to know the transition coefficients $(e_m:g_k)$ with $m\ge k\ge0$ and $m\ge1$. They are computed in the next theorem. 

\begin{theorem}\label{thm5.A}

{\rm(i)} For $m\ge1$ and $k\ge1$,
\begin{gather*}
(e_m:g_k)=2(L+\epsi+m-1)(2L+2\epsi+m-1)_{k-1}\frac{(m-1)!}{(m-k)!}\\
\times \frac{(a+1)_k}{(L)_k(L+a)_k(k+2\epsi)_k}\\
\times
{}_4F_3\left[\begin{matrix}  k-m,\; k+1,\; k+a+1,\; 2L+2\epsi+m+k-2
\\ L+k,\; L+a+k,\; 2k+2\epsi+1\end{matrix}\;\Biggl|\;1\right].
\end{gather*}

{\rm(ii)} For $m\ge1$ and $k=0$,
\begin{equation*}
\begin{gathered}
(e_m:g_0)=-\frac{2(L+\epsi+m-1)(a+1)}{L(L+a)(2\epsi+1)}\\
\times{}_4F_3\left[\begin{matrix}  1-m,\; 1,\;a+2,\; 2L+2\epsi+m-1
\\ L+1,\; L+a+1,\; 2\epsi+2\end{matrix}\;\Biggl|\;1\right].
\end{gathered}
\end{equation*}
\end{theorem}

This theorem (together with Proposition \ref{prop5.A}) is a detailed version of Theorem \ref{thmC} from section \ref{sect1.9}. 

The proof is based on three lemmas. To state them we need to introduce auxiliary rational functions
$$
f_\ell(t):=\frac1{(-t+L+\epsi)_\ell(t+L+\epsi)_\ell}, \quad \ell\in\Z_{\ge0}.
$$

From the proof of Proposition \ref{prop5.A} we know that the transition matrix between the bases $\{e_m\}$ and $\{g_k\}$ is triangular with respect to the natural order on the index set $\Z_{\ge0}$. 

\begin{lemma}\label{lemma5.C}

{\rm(i)} The functions $f_\ell$ form a basis of $\FF(\epsi,L)$.

{\rm(ii)} The transition matrices between all three bases, $\{e_m\}$, $\{g_k\}$, and $\{f_\ell\}$, are triangular.  
\end{lemma}

\begin{proof}
Note that $f_0=1$. Next, if $\ell\ge1$, then the function $f_\ell(t)$ vanishes at infinity and its singularities are precisely simple poles at the points $\pm A_m$ with $m=1,\dots,\ell$. It follows that the functions $f_\ell$ lie in the space $\FF(\epsi,L)$. The same argument also shows that the transition coefficients $(f_\ell:e_m)$ vanish unless $m\le\ell$. Moreover, $(f_\ell:e_m)\ne0$ for $m=\ell$. This means that $\{f_\ell\}$ is a basis and the transition matrix between $\{f_\ell\}$ and $\{e_m\}$ is triangular. This in turn implies that all transition matrices in question are triangular, too. 
\end{proof}

We write $(\phi:f_\ell)$ for the coefficients of the expansion of a function $\phi\in\FF(\epsi,L)$ in the basis $\{f_\ell\}$. From Lemma \ref{lemma5.C} we have 
\begin{equation}\label{5.C}
(e_m:g_k)=\sum_{\ell=k}^m(e_m:f_\ell)(f_\ell:g_k), \quad m\ge k. 
\end{equation}
The purpose of the two next  lemmas is to compute the coefficients $(e_m:f_\ell)$ and $(f_\ell:g_k)$.

\begin{lemma}\label{lemma5.A}
Let $m\in\Z_{\ge1}$. 

{\rm(i)} We have 
\begin{equation}
(e_m:f_0)=0.
\end{equation}

{\rm(ii)} For $\ell\ge1$, we have
\begin{gather*}
(e_m:f_\ell)=2(-1)^{\ell}(L+\epsi+m-1)\prod_{j=1}^{\ell-1}(2L+2\epsi+m+j-2)(m-j).
\end{gather*}
\end{lemma}

Note that the triangularity property is ensured by the product $\prod_{j=1}^{\ell-1}(m-j)$. 

\begin{proof}
(i) The functions $e_m(t)$ and $f_\ell(t)$ with nonzero indices vanish at $t=\infty$. This implies (i). 

(ii) Let $z$ and $a_1,a_2,\dots$ be formal variables; the next identity is easily proved by induction on $m$:
$$
\frac1{z-a_m}=\frac1{z-a_1}+\frac{a_m-a_1}{(z-a_1)(z-a_2)}+\dots+\frac{(a_m-a_1)\dots(a_m-a_{m-1})}{(z-a_1)\dots(z-a_m)}.
$$
That is, the coefficients in the expansion are 
\begin{equation}\label{5.A.1}
\left(\frac1{z-a_m}:\frac1{(z-a_1)\dots(z-a_\ell)}\right)=\prod_{j=1}^{\ell-1}(a_m-a_j), \qquad m=1,2,\dots\,.
\end{equation}

Observe now that
$$
e_m(t)=\frac{2(L+\epsi+m-1)}{t^2-(L+\epsi+m-1)^2}, \qquad m=1,2,\dots,
$$
and
$$
f_\ell(t)=\frac{(-1)^\ell}{(t^2-(L+\epsi)^2)\dots(t^2-(L+\epsi+\ell-1)^2)}.
$$
So we set
$$
z=t^2, \qquad a_m=(L+\epsi+m-1)^2
$$
and apply \eqref{5.A.1}. This proves (ii). 
\end{proof}

\begin{lemma}\label{lemma5.B}
The following formula holds.
\begin{equation*}
(f_\ell:g_k)=\frac{2(k+\epsi)\Ga(a+\ell+1)\Ga(k+2\epsi)(-\ell)_k}{(L)_\ell(L+a)_\ell\Ga(a+1)\Ga(k+2\epsi+\ell+1)k!}.
\end{equation*}
\end{lemma}

Note that the triangularity property is ensured by the factor $(-\ell)_k$. 

\begin{proof}
The functions $g_k(t)$ can be written in the form
\begin{equation}\label{eq5.D}
g_k(t):=\sum_{\ell=0}^k\frac{(-k)_\ell(k+2\epsi)_\ell}{(a+1)_\ell \ell!}\,\wt f_\ell(t), \quad k\in\Z_{\ge0},
\end{equation}
where
$$
\wt f_\ell(t):=(L)_\ell(L+a)_\ell f_\ell(t).
$$
Compare \eqref{eq5.D} with the well-known formula for the Jacobi polynomials (Erdelyi \cite[section 10.8]{Erdelyi})
$$
\wt P^{(a,b)}_k(x):=\frac{\Ga(a+1)k!}{\Ga(k+a+1)}\,P^{(a,b)}_k(x)=\sum_{\ell=0}^k\frac{(-k)_\ell(k+2\epsi)_\ell}{(a+1)_\ell \ell!}\left(\frac{1-x}2\right)^\ell.
$$
The coefficients in these two expansions are the same, which implies that the desired coefficients $(\wt f_\ell:g_k)$ coincide with the coefficients in the expansion of $(\frac12(1-x))^\ell$ in the polynomials $\wt P^{(a,b)}_k(x)$. This expansion can be easily derived using the Rodrigues formula for the Jacobi polynomials:  
\begin{gather}
\left(\frac{1-x}2\right)^\ell=\sum_{k=0}^\ell\frac{2(k+\epsi)\Ga(a+\ell+1)\Ga(k+2\epsi)(-\ell)_k}{\Ga(k+a+1)\Ga(k+\ell+2\epsi+1)}\,P^{(a,b)}_k(x)\label{5.B.1}\\
=\sum_{k=0}^\ell\frac{2(k+\epsi)\Ga(a+\ell+1)\Ga(k+2\epsi)(-\ell)_k}{\Ga(a+1)\Ga(k+\ell+2\epsi+1)k!}\,\left(\frac{\Ga(a+1)k!}{\Ga(k+a+1)}\,P^{(a,b)}_k(x)\right).\label{5.B.2}
\end{gather}
Formula \eqref{5.B.1} can be found in handbooks, see \cite[5.12.2.1]{Brychkov} (in the Russian edition (2006), it is 5.11.2.5) and \cite[10.20 (3)]{Erdelyi} (but note that the expression in the latter reference contains a typo: namely,  the factor $\Ga(2n+\al+\be+1)$ should be replaced by $2n+\al+\be+1$).

From \eqref{5.B.2} we obtain
$$
(f_\ell:g_k)=\dfrac{(\wt f_\ell:g_k)}{(L)_\ell(L+a)_\ell}=\frac{2(k+\epsi)\Ga(a+\ell+1)\Ga(k+2\epsi)(-\ell)_k}{(L)_\ell(L+a)_\ell\Ga(a+1)\Ga(k+\ell+2\epsi+1)k!},
$$
which is the desired result. 
\end{proof}

\begin{proof}[Proof of Theorem \ref{thm5.A}]

(i) We suppose $m\ge1$ and $k\ge 1$, and set $\ell=k+n$. 

Let us rewrite the formulas of Lemmas \ref{lemma5.A} and \ref{lemma5.B}:
\begin{gather*}
(e_m:f_\ell)=2(-1)^{\ell}(L+\epsi+m-1)\prod_{j=1}^{\ell-1}(2L+2\epsi+m+j-2)(m-j)\\
=2(-1)^{k}(L+\epsi+m-1)(2L+2\epsi+m-1)_{k-1}\frac{(m-1)!}{(m-k)!}\\
\times(2L+2\epsi+m+k-2)_n(k-m)_n
\end{gather*}
and
\begin{gather*}
(f_\ell:g_k)=\frac{2(k+\epsi)\Ga(a+\ell+1)\Ga(k+2\epsi)(-\ell)_k}{(L)_\ell(L+a)_\ell\Ga(k+a+1)\Ga(k+2\epsi+\ell+1)}\notag\\
=\frac{\Ga(k+2\epsi)(-1)^k(a+1)_k}{(L)_k(L+a)_k\Ga(2k+2\epsi)}\\
\times\frac{(a+k+1)_n(k+1)_n}{(L+k)_n(L+a+k)_n(2k+2\epsi+1)_nn!}
\end{gather*}

Next, put aside the factors that do not depend on $n$ and then take the sum over $n=0,\dots,m-k$, which corresponds to the summation over $\ell=k,\dots,m$ in \eqref{5.C}.  It gives
\begin{gather*}
\sum_{n=0}^{m-k}\frac{(a+k+1)_n(k+1)_n(2L+2\epsi+m+k-2)_n(k-m)_n}{(L+k)_n(L+a+k)_n(2k+2\epsi+1)_nn!}\\
={}_4F_3\left[\begin{matrix}  k-m,\; k+1,\; k+a+1,\; 2L+2\epsi+m+k-2
\\ L+k,\; L+a+k,\; 2k+2\epsi+1\end{matrix}\;\Biggl|\;1\right].
\end{gather*}
Taking into account the remaining factors gives the desired expression.

(ii) We suppose $m\ge1$ and $k=0$. The computation is similar to the previous one. Since $(e_m:f_0)=0$, we have
$$
(e_m:g_0)=\sum_{\ell=1}^m(e_m:f_\ell)(f_\ell:g_0).
$$
It is convenient to set $n:=\ell-1$, so that $n$ ranges from $0$ to $m-1$. The lemmas show that
$$
(e_m:f_{n+1})=-2(L+m-1)(2L+2\epsi)_n(1-m)_n
$$
and
\begin{gather*}
(f_{n+1}:g_0)=\frac{2\epsi\Ga(2\epsi)\Ga(a+n+2)}{(L)_{n+1}(L+a)_{n+1})\Ga(a+1)\Ga(2\epsi+n+2)}\\
=\frac{a+1}{L(L+a)(2\epsi+1)}\cdot\frac{(a+2)_n}{(L+1)_n(L+a+1)_n(2\epsi+2)_n}.
\end{gather*}
It follows
$$
(e_m:g_0)=-\frac{2(L+\epsi+m-1)(a+1)}{L(L+a)(2\epsi+1)}
\sum_{n=0}^{m-1}\frac{(1-m)_n(2L+2\epsi+m-1)_n(a+2)_n}{(L+1)_n(L+a+1)_n(2\epsi+2)_n},
$$
which is the desired expression.
\end{proof}

\section{Elementary expression of functions $g_k$ in the case of symplectic and orthogonal characters}\label{sect6}

Recall that we are dealing with the functions defined by \eqref{eq1.g_k}:
\begin{equation*}
g_k(t;a,\epsi,L):={}_4F_3\left[\begin{matrix}-k,\, k+2\epsi,\, L,\, L+a\\-t+L+\epsi,\,
t+L+\epsi,\, a+1\end{matrix}\,\Biggl|\,1\right], \qquad k=0,1,2,\dots\,.
\end{equation*}

In this section we consider the three special cases
$$
(a,\epsi)= (\tfrac12,1), \; (\tfrac12, \tfrac12), \; (-\tfrac12,0),
$$
which correspond to the series $\CC$, $\BB$, $\DD$, and we introduce the alternate notation
\begin{gather*}
g_k^{(\CC)}(t;L):=g_k(t; \tfrac12, 1,L)={}_4F_3\left[\begin{matrix}-k,\, k+2,\, L,\, L+\tfrac12\\-t+L+1,\,
t+L+1,\, \tfrac32\end{matrix}\,\Biggl|\,1\right], \\
g_k^{(\BB)}(t;L):=g_k(t; \tfrac12, \tfrac12,L)={}_4F_3\left[\begin{matrix}-k,\, k+1,\, L,\, L+\tfrac12\\-t+L+\tfrac12,\,
t+L+\tfrac12,\, \tfrac32\end{matrix}\,\Biggl|\,1\right], \\
g_k^{(\DD)}(t;L):=g_k(t; -\tfrac12, 0,L)={}_4F_3\left[\begin{matrix}-k,\, k,\, L,\, L-\tfrac12\\-t+L,\,
t+L,\, \tfrac12\end{matrix}\,\Biggl|\,1\right], 
\end{gather*}

\begin{theorem}\label{thm6.A}
These three hypergeometric series admit closed elementary expressions{\rm:}  

\begin{equation}\label{eq6.F1}
g^{(\CC)}_k(t;L)
=\frac1{2(k+1)(1-2L)t}\left[\frac{(t-L)_{k+2}}{(t+L+1)_k}
-\frac{(-t-L)_{k+2}}{(-t+L+1)_k}\right].
\end{equation}

\begin{equation}\label{eq6.F2}
g^{(\BB)}_k(t;L)
=\frac1{2(k+\tfrac12)(1-2L)}\left[\frac{(t-L+\frac12)_{k+1}}{(t+L+\frac12)_k}
+\frac{(-t-L+\frac12)_{k+1}}{(-t+L+\frac12)_k}\right].
\end{equation}

\begin{equation}\label{eq6.F3}
g^{(\DD)}_k(t;L)
=\frac12\left[\frac{(t-L+1)_k}{(t+L)_k}+\frac{(-t-L+1)_k}{(-t+L)_k}\right].
\end{equation}
\end{theorem}

\begin{remark} In the very beginning of this work I expected that the summation formulas for the series $g^{(\CC)}_k(t;L)$, $g^{(\BB)}_k(t;L)$, and $g^{(\DD)}_k(t;L)$ should exist, so I first tried to find them in the literature. As a result of the search I managed to find the first formula in the handbook \cite{PBM} (p. 556 of the Russian original, eq. 7.5.1 (42)). Unfortunately, the handbook did not provide any proof or a suitable reference. Then I asked Eric Rains, and he kindly communicated to me a unified derivation of all three formulas in the letter \cite{Rains}. I am very grateful to him for this help. I reproduce below his argument in a more detailed form, but with a different proof for the next lemma. 
\end{remark}

\begin{lemma}
{\rm(i)} Suppose $A$ is an even nonpositive integer. Then
\begin{equation}\label{eq6.C1}
{}_3F_2\left[\begin{matrix}A,\, B,\, D\\ \half(A+B+1), \, E \end{matrix}\,\Biggl|\,1\right]=
{}_4F_3\left[\begin{matrix}\frac12 A,\, \half B,\, E-D, \, D\\ \half(A+B+1), \, \half E, \half(E+1) \end{matrix}\Biggl|1\right].
\end{equation}

{\rm(ii)} Suppose $A$ is an odd negative integer. Then
\begin{equation}\label{eq6.C2}
{}_3F_2\left[\begin{matrix}A,\, B,\, D\\ \half(A+B+1), \, E \end{matrix}\,\Biggl|\,1\right]=
\frac{E-2D}E\,{}_4F_3\left[\begin{matrix}\half(A+1),\, \half(B+1),\, E-D, \, D\\ \half(A+B+1), \, \half(E+1), \half E +1 \end{matrix}\,\Biggl|\,1\right].
\end{equation}
\end{lemma}

\begin{proof}[Proof of the lemma]
(i) This is formula (3.6) in Krattenthaler--Rao \cite{KR}. As explained there, it is obtained from the Gauss' quadratic transformation formula \cite[section 2.11, eq. (2)]{Erdelyi}
\begin{equation}\label{eq6.A}
{}_2F_1\left[\begin{matrix}A,\, B\\ \half(A+B+1) \end{matrix}\,\Biggl|\,z\right]
={}_2F_1\left[\begin{matrix}\half A,\, \half B\\ \half(A+B+1) \end{matrix}\,\Biggl|\,4z(1-z)\right]
\end{equation}
by the following simple procedure: (1) convert the hypergeometric series on both sides into (finite) sums;
(2) multiply both sides of the equation by the factors $z^{d-1}(1-z)^{e-d-1}$;
(3) integrate term by term with respect to $z$ for $0\le z\le 1$;
(4) interchange integration and summation;
(5) use the beta integral to evaluate the integrals inside the summations;
(6) convert the sums back into hypergeometric notation.

(ii) We replace  \eqref{eq6.A} with another quadratic transformation formula \cite[section 2.11, eq. 19]{Erdelyi}:
\begin{equation}\label{eq6.B}
{}_2F_1\left[\begin{matrix}A,\, B\\ \half(A+B+1) \end{matrix}\,\Biggl|\,z\right]
=(1-2z)\,{}_2F_1\left[\begin{matrix}\half(A+1),\, \half(B+1)\\ \half(A+B+1) \end{matrix}\,\Biggl|\,4z(1-z)\right],
\end{equation}
write $1-2z=(1-z)-z$, and apply the same procedure.  

Note also that \eqref{eq6.B} can be easily obtained from \eqref{eq6.A} by differentiating over $z$. 
\end{proof}

\begin{proof}[Proof of Theorem \ref{thm6.A}]
We apply  the well-known transformation formula
\begin{multline*}
{}_4F_3\left[\begin{matrix}-k, \al_1,\al_2,\al_3\\\be_1,\be_2,\be_3 \end{matrix}\,\Biggl|\,1\right]
=\frac{(-1)^k(\al_1)_k(\al_2)_k(\al_3)_k}{(\be_1)_k(\be_2)_k(\be_3)_k}\\
\times {}_4F_3\left[\begin{matrix}-k, \; -\be_1-k+1,\; -\be_2-k+1,\; -\be_3-k+1\\ -\al_1-k+1,\; -\al_2-k+1,\; -\al_3-k+1 \end{matrix}\,\Biggl|\,1\right],
\end{multline*}
which is obtained by summing a terminating series in the reverse order.  This gives 
\begin{align}
g_k(t;a,\epsi,L)&=\frac{(-1)^k(k+2\epsi)_k(L)_k(L+a)_k}{(-t+L+\epsi)_k(t+L+\epsi)_k(a+1)_k} \label{eq6.D1}\\
&\times {}_4F_3\left[\begin{matrix}-k,\; -k-a,\; t-L-\epsi-k+1,\; -t-L-\epsi-k+1\\
-2k-2\epsi+1,\;  -L-k+1,\;  -L-k-a+1\end{matrix}\,\Biggl|\,1\right].\label{eq6.D2}
\end{align}
In the three special cases under consideration the hypergeometric series on the right-hand side has the same form as in \eqref{eq6.C1} or \eqref{eq6.C2}. However, the application of these formulas requires some caution, as it will soon become clear. 
Let us examine the three cases separately.  

Case ($\CC$): $a=\frac12$, $\epsi=1$. The ${}_4 F_3$ in \eqref{eq6.D2} has the same form as on the right-hand side of \eqref{eq6.C2}, with the parameters
\begin{equation*}
A=-2k-1, \quad B=-2k-2, \quad D=-t-L-k, \quad E=-2L-2k,
\end{equation*}
so that the corresponding ${}_3 F_2$ on the left-hand side is
\begin{equation}\label{eq6.L}
{}_3F_2\left[\begin{matrix}-2k-1,\, -2k-2,\, -t-L-k\\ -2k-1, \, -2L-2k \end{matrix}\,\Biggl|\,1\right].
\end{equation}
How to interpret this expression? It is tempting to directly reduce it to a ${}_2 F_1$ series by removing the parameter $-2k-1$ from the upper and lower rows, but this would give an incorrect result. A correct argument is the following.

We come back to the initial ${}_4F_3$ series in \eqref{eq6.C2}, keep $a$ as a parameter, but exclude  $\epsi$ by imposing the linear  relation
$$
-2k-2\epsi+1=(-k)+(-k-a)-\tfrac12, \quad \text{that is,} \quad  2\epsi=a+\tfrac32.
$$
The point is that the resulting ${}_4F_3$ series is still of the form \eqref{eq6.C2}. Next, it is a rational function of the parameter $a$ and has no singularity at $a=\frac12$ for generic  $L$. Therefore, we may apply the identity \eqref{eq6.C2} and then pass to the limit as $a$ goes to $\frac12$. This leads to the conclusion that \eqref{eq6.L} must be interpreted as
\begin{equation}\label{eq6.E}
{}_2F_1\left[\begin{matrix}-2k-2,\, -t-L-k\\ -2L-2k \end{matrix}\,\Biggl|\,1\right] - (\text{the last term of the series expansion}).
\end{equation}

Applying the Chu-Vandermonde identity \cite[Corollary 2.2.3]{AAR}
\begin{equation}\label{eq6.J}
{}_2F_1\left[\begin{matrix}-N,\, \al\\ \be \end{matrix}\,\Biggl|\,1\right]=\frac{(\be-\al)_N}{(\be)_N}, \qquad N=0,1, 2,\dots, \quad \be\ne0,-1\dots, -N,
\end{equation}
we see that \eqref{eq6.E} is equal to
\begin{equation*}
\frac{(t-L-k)_{2k+2}-(-t-L-k)_{2k+2}}{(-2L-2k)_{2k+2}}.
\end{equation*}

Next, we have to multiply this by 
\begin{equation*}
\frac{(-1)^k(k+2)_k(L)_k(L+\tfrac12)_k}{(-t+L+1)_k(t+L+1)_k(\tfrac32)_k}\cdot \frac{-2L-2k}{2t},
\end{equation*}
where the first fraction comes from \eqref{eq6.D1} and the second fraction comes from $\dfrac{E}{E-2D}$ (see \eqref{eq6.C2}). After a simplification this finally gives the desired expression \eqref{eq6.F1}. 

Case ($\BB$): $a=\frac12$, $\epsi=\frac12$.  Now the ${}_4 F_3$ in \eqref{eq6.D2} has the same form as on the right-hand side of \eqref{eq6.C1}, with the parameters
\begin{equation*}
A=-2k, \quad B=-2k-1, \quad D=-t-L-k+\tfrac12, \quad E=-2L-2k+1,
\end{equation*}
so that the corresponding ${}_3 F_2$ on the left-hand side is
\begin{equation*}
{}_3F_2\left[\begin{matrix}-2k,\, -2k-1,\, -t-L-k+\frac12\\ -2k, \, -2L-2k+1 \end{matrix}\,\Biggl|\,1\right].
\end{equation*}
By the same argument as above, the correct elimination of the parameter $-2k$ leads to
\begin{equation*}
{}_2F_1\left[\begin{matrix}-2k-1,\, -t-L-k+\frac12\\ -2L-2k+1 \end{matrix}\,\Biggl|\,1\right] - (\text{the last term of the series expansion}).
\end{equation*}
Applying the Chu--Vandermonde identity \eqref{eq6.J} to this ${}_2F_1$ we obtain
\begin{equation*}
\frac{(t-L-k+\tfrac12)_{2k+1}+(-t-L-k+\tfrac12)_{2k+1}}{(-2L-2k+1)_{2k+1}}.
\end{equation*}
Next, we multiply this by 
\begin{equation*}
\frac{(-1)^k(k+1)_k(L)_k(L+\tfrac12)_k}{(-t+L+\frac12)_k(t+L+\frac12)_k(\tfrac32)_k}, 
\end{equation*}
and after a simplification this finally gives the desired expression \eqref{eq6.F2}.

Case ($\DD$): $a=-\frac12$, $\epsi=0$.  The ${}_4 F_3$ in \eqref{eq6.D2} still has the same form as on the right-hand side of \eqref{eq6.C1}, with the parameters
\begin{equation*}
A=-2k, \quad B=-2k+1, \quad D=-t-L-k+1, \quad E=-2L-2k+2.
\end{equation*}
According to \eqref{eq6.C1} this leads  to the ${}_3 F_2$ series
\begin{equation*}
{}_3F_2\left[\begin{matrix}-2k,\, -2k+1,\, -t-L-k+1\\ -2k+1, \, -2L-2k+2 \end{matrix}\,\Biggl|\,1\right].
\end{equation*}
Here the correct elimination of the parameter $-2k+1$ is achieved by the limit transition
\begin{equation}\label{eq6.G}
\lim_{a\to-\frac12}{}_3F_2\left[\begin{matrix}-2k,\, -2k-2a,\, -t-L-k+\frac34-\tfrac12 a\\ -2k-a+\tfrac12, \, -2L-2k+2 \end{matrix}\,\Biggl|\,1\right].
\end{equation}
The limit in \eqref{eq6.G} can be taken term-wise. It differs from the expansion of 
\begin{equation}\label{eq6.H}
{}_3F_2\left[\begin{matrix}-2k,\, -t-L-k+1\\ -2L-2k+2 \end{matrix}\,\Biggl|\,1\right]
\end{equation}
in the last term only. Namely, in \eqref{eq6.G}, the last term is
\begin{equation*}
\lim_{a\to-\frac12}\frac{(-2k)_{2k}(-2k-2a)_{2k}(-t-L-k+\tfrac34-\tfrac12 a)}{(-2k-a+\tfrac12)_{2k}(-2L-2k+1)_{2k}(2k)!}=
2\, \frac{(-t-L-k+1)_{2k}}{(-2L-2k+2)_{2k}}.
\end{equation*}
while the last term in \eqref{eq6.H} is 
\begin{equation}\label{eq6.I}
\frac{(-t-L-k+1)_{2k}}{(-2L-2k+2)_{2k}}.
\end{equation}
We conclude that the limit in \eqref{eq6.G} is equal to the sum of \eqref{eq6.H} and \eqref{eq6.I}:
\begin{equation}\label{eq6.K}
{}_3F_2\left[\begin{matrix}-2k,\, -t-L-k+1\\ -2L-2k+2 \end{matrix}\,\Biggl|\,1\right]
+\frac{(-t-L-k+1)_{2k}}{(-2L-2k+2)_{2k}}.
\end{equation}

Using the Chu-Vandermonde identity \eqref{eq6.J} we obtain that \eqref{eq6.K} equals
\begin{equation*}
\frac{(t-L-k+1)_{2k}+(-t-L-k+1)_{2k}}{(-2L-2k+2)_{2k}}.
\end{equation*}
Next, we multiply this by 
\begin{equation*}
\frac{(-1)^k(k)_k(L)_k(L-\tfrac12)_k}{(-t+L)_k(t+L)_k(\tfrac12)_k}, 
\end{equation*}
and after a simplification this finally gives the desired expression \eqref{eq6.F3}. 
\end{proof}

\section{Elementary expression of transition coefficients $(e_m:g_k)$ in the case of symplectic and orthogonal characters}\label{sect7}

Let
\begin{equation*}
E(m,k)=E(m,k;a,\epsi,L)
\end{equation*}
denote the explicit expression for the transition coefficients $(e_m:g_k)$ obtained in Theorem \ref{thm5.A}.

As in section \ref{sect6}, we examine the three distinguished cases when the parameters $(a,\epsi)$ correspond to the $\CBD$ characters. We show that then the formulas for the coefficients $E(m,k)=E(m,k;a,\epsi,L)$ are simplified: the  ${}_4F_3$ hypergeometric series that appear in the formulas admit closed elementary expressions.  To distinguish between these three cases $\CBD$ we use the superscript $(\CC)$, $(\BB)$ or $(\DD)$, respectively.

\begin{theorem}\label{thm7.A}

{\rm (i)} If $m\ge k\ge1$, then
$$
E^{(\CC)}(m,k):=E(m,k;\tfrac12,1,L)=\frac{2(k+1)(m-1)!(2L-2)(2L-1)(L+m)(2L+m-k-3)!}
{(m-k)!(2L+m)!}.
$$
$$
E^{(\BB)}(m,k):=E(m,k;\tfrac12,\tfrac12,L)=\frac{2(k+\tfrac12)(m-1)!(2L-2)(2L-1)(2L+m-k-3)!}{(m-k)!(2L+m-1)!},
$$
$$
E^{(\DD)}(m,k):=E(m,k;-\tfrac12,0,L)=\frac{2(m-1)!(2L-2)(2L+m-k-3)!}{(m-k)!(2L+m-2)!},
$$

{\rm(ii)} If $m\ge1$ and $k=0$, then
$$
E^{(\CC)}(m,0):=E(m,0;\tfrac12,1,L)=-\frac{2(m+4L-3)(L+m)}{(2L+m)(2L+m-1)(2L+m-2)}.
$$
$$
E^{(\BB)}(m,0):=E(m,0;\tfrac12,\tfrac12,L)=-\frac{2m+6L-5}{(2L+m-1)(2L+m-2)},
$$
$$
E^{(\DD)}(m,0):=E(m,0;-\tfrac12,0,L)=-\frac{2}{2L+m-2},
$$
\end{theorem}

Recall that $E(m,k)=0$ for $k>m$, so that the theorem covers all possible cases. The proof is based on the following lemma.

\begin{lemma}\label{lemma7.A}
Let $n=0,1,2,\dots$\,. The following two formulas hold:
\begin{equation}\label{eq7.A1}
{}_4F_3\left[\begin{matrix}-n,\, A,\, A+\frac12,\, 2B+n
\\B,\,  B+\frac12,\, 2A+1\end{matrix}\Biggl|1\right]=\frac{\Ga(2B-2A+n)\Ga(2B)}{\Ga(2B-2A)\Ga(2B+n)}
\end{equation}
and
\begin{equation}\label{eq7.A2}
{}_4F_3\left[\begin{matrix}-n,\, A+\frac12,\, A+1,\, 2B+n
\\B+\frac12,\,  B+1,\, 2A+1\end{matrix}\Biggl|1\right]=\frac{B}{(B+n)}\,\frac{\Ga(2B-2A+n)\Ga(2B)}{\Ga(2B-2A)\Ga(2B+n)}.
\end{equation}
\end{lemma}

\begin{proof}[Proof of the lemma]
For the first formula, see Slater \cite[p. 65, (2.4.2.2) and (III.20)]{Slater}. The second formula is derived from the first one using the transformation \cite[(2.4.1.7)]{Slater}, which holds for any balanced terminating series ${}_4F_3(1)$. In a  slightly rewritten notation it reads:
\begin{equation}
{}_4F_3\left[\begin{matrix}-n,\, a_1,\, a_2,\, x
\\b_1,\,  b_2,\, y\end{matrix}\Biggl|1\right]=\frac{(b_1-x)_n(b_1-u)_n}{(b_1)_n(b_1-x-u)_n} {}_4F_3\left[\begin{matrix}-n,\, a_1-u,\, a_2-u,\, x
\\b_1-u,\, b_2-u,\, y\end{matrix}\Biggl|1\right],
\end{equation}
where 
$$
u:=a_1+a_2-y.
$$
We specialize it to
$$
a_1=A+\tfrac12, \quad  a_2=A+1, \quad  x=2B+n, \quad b_1=B+\tfrac12, \quad  b_2=B+1, \quad y=2A+1.
$$

For other derivations, see \cite[(3.20) and (3.21)]{GesselStanton} and further references therein.
\end{proof}

\begin{proof}[Proof of Theorem \ref{thm7.A}]
(i) Let us show that Lemma \ref{lemma7.A} can be applied to the ${}_4F_3$ series displayed in claim (ii) of Theorem  \ref{thm5.A}. 

Indeed, the series in question is
$$
{}_4F_3\left[\begin{matrix}  k-m,\; k+1,\; k+a+1,\; 2L+2\epsi+m+k-2
\\ L+k,\; L+a+k,\; 2k+2\epsi+1\end{matrix}\;\Biggl|\;1\right].
$$
We look at the triple $(k+1, k+a+1, 2k+2\epsi+1)$. It takes the following form:
$$
(k+1, k+a+1, 2k+2\epsi+1)=\begin{cases}
(k+1, k+\tfrac32, 2k+3), & \text{case $(C)$}, \\
(k+1, k+\tfrac32, 2k+2), & \text{case $(B)$}, \\
(k+1, k+\tfrac12, 2k+1), & \text{case $(D)$}.
\end{cases}
$$
It follows that formula \eqref{eq7.A1} is applicable in the cases $(C)$ and $(D)$, while formula \eqref{eq7.A2} is applicable in the case $(B)$. This leads to the expressions in claim (i).

(ii) Formulas \eqref{eq7.A1} and \eqref{eq7.A2} are not applicable to the ${}_4F_3$  series in claim (iii) of Theorem  \ref{thm5.A} in our three special cases. Perhaps suitable summation formulas can be extracted from the literature, but I have not succeeded.  Here is another way to solve the problem. 

Namely, observe that $e_m(\infty)=0$ for each $m\ge1$ while $g_k(\infty)=1$ for all $k$. It follows that 
$$
E(m,0)=-\sum_{k=1}^m E(m,k), \qquad m=1,2,\dots\,.
$$
Therefore, the formulas in claim (ii) are equivalent to the following three identities
\begin{multline*}
\sum_{k=1}^m \frac{(2k+2)(m-1)!(2L-2)(2L-1)(L+m)(2L+m-k-3)!}
{(m-k)!(2L+m)!}\\
=\frac{2(m+4L-3)(L+m)}{(2L+m)(2L+m-1)(2L+m-2)}.
\end{multline*}

\begin{multline*}   
\sum_{k=1}^m\frac{(2k+1)(m-1)!(2L-2)(2L-1)(2L+m-k-3)!}{(m-k)!(2L+m-1)!}\\
=\frac{2m+6L-5}{(2L+m-1)(2L+m-2)}.
\end{multline*}

\begin{equation*}
\sum_{k=1}^m \frac{2(m-1)!(2L-2)(2L+m-k-3)!}{(m-k)!(2L+m-2)!}
=\frac{2}{2L+m-2}
\end{equation*}

Set $M:=2L-2$; after simplification one can rewrite these identities as follows
$$
S^{(\CC)}(m,M):=M\sum_{k=1}^m
\frac{(M+m-k-1)!(k+1)}{(m-k)!}=\frac{(M+m-1)!(m+2M+1)}{(m-1)!(M+1)}.
$$

$$
S^{(\BB)}(m,M):=M\sum_{k=1}^m
\frac{(M+m-k-1)!(2k+1)}{(m-k)!}=\frac{(M+m-1)!(2m+3M+1)}{(l-1)!(M+1)},
$$

$$
S^{(\DD)}(m,M):=M\sum_{k=1}^m
\frac{(M+m-k-1)!}{(m-k)!}=\frac{(M+m-1)!}{(m-1)!},
$$

The  identity for $S^{(\DD)}(m,M)$ is checked by induction on
$m$ using the relation
$$
S^{(\DD)}(m+1,M)=S^{(\DD)}(m,M)+\frac{M(M+m-1)!}{m!}. 
$$

Next, the other two sums, $S^{(\BB)}(m,M)$ and $S^{(\CC)}(m,M)$, are
reduced to $S^{(\DD)}(m,M)$ using the relations
$$
(2m+1)S^{(\DD)}(m,M)-S^{(\BB)}(m,M)
=2M\sum_{k=0}^{m-1}\frac{(M+m-k-1)!}{(m-k-1)!)}
=\frac{2M}{M+1}S^{(\DD)}(m-1,M+1)
$$
and
$$
(m+1)S^{(\DD)}(m,M)-S^{(\CC)}(m,M)
=M\sum_{k=0}^{m-1}\frac{(M+m-k-1)!}{(m-k-1)!)}
=\frac{M}{M+1}S^{(\DD)}(m-1,M+1).
$$
This completes the proof.
\end{proof}

\section{Proof of Theorem \ref{thmD} and application to discrete splines}\label{sect8}

\subsection{Proof of Theorem \ref{thmD}}\label{sect8.1}

The results of computations in Sections \ref{sect6} and \ref{sect7} are summarized below in Theorem \ref{thm8.A} (a detailed version of Theorem \ref{thmD} from section \ref{sect1.9}). To state it we recall  the relevant definitions and notation. 

\begin{itemize}

\item 

We are dealing with the stochastic matrix $\La^N_K$ with the entries $\La^N_K(\nu,\ka)$, where $\nu$ ranges over $\Sign^+_N$,  $\ka$ ranges over $\Sign^+_K$, and $N>K\ge1$; see Definition \ref{def1.A}. In general, the matrix depends on a pair $(a,b)$ of Jacobi parameters, but it  is convenient to replace the second parameter $b$ by $\epsi:=\frac12(a+b+1)$. 

\item 
We are especially interested in the three distinguished cases, which are linked to the characters of type $\CC$, $\BB$, and $\DD$. In terms of the parameters $(a,\epsi)$, this means that  
\begin{equation}\label{eq8.A}
(a,\epsi)=\begin{cases} (\tfrac12,1),\\ (\tfrac12,\tfrac12),\\ (-\tfrac12,0), \end{cases}
\end{equation} 
respectively. 

\item

We set $L:=N-K+1$, so that $L$ is an integer $\ge2$. In \eqref{eq1.N} we introduced a grid $\AAA(\epsi,L)$ on $\R_{>0}$ depending on $\epsi$ and $L$:
 \begin{equation*}
\AAA(\epsi,L):=\{A_1,A_2,\dots\}, \qquad A_m:=L+\epsi+m-1, \quad m=1,2,\dots,
\end{equation*}

\item

In section \ref{sect1.8} we introduced the space $\FF(\epsi,L)$ formed by even rational functions with simple poles located at $(-\AAA(\epsi,L))\cup\AAA(\epsi,L)$. For a rational function $\phi\in \FF(\epsi,L)$, we denote by $\Res_{t=A_m}(\phi(t))$ its residue at the point $t=A_m$ of the grid $\AAA(\epsi,L)$. 
 
\item

In \eqref{eq1.g_k} we introduced even rational functions $g_k(t)=g_k(t;a,\epsi,L)$ with index $k=0,1,2,\dots$ and general parameters $(a,\epsi)$. In the general case, $g_k(t)$ is given by terminating hypergeometric series ${}_4F_3$. For the special values \eqref{eq8.A},  these functions admit an explicit elementary expression (Theorem \ref{thm6.A}). 

\item

In Theorem \ref{thm5.A} we computed the transition coefficients $(e_m:g_k)$, renamed to $E(m,k)$ in the beginning of section \ref{sect7}.  For the special values \eqref{eq8.A},  these coefficients admit an explicit elementary expression (Theorem \ref{thm7.A}). 

\item

In  \eqref{eq1.F_N} we assigned to each signature $\nu\in\Sign^+_N$ its characteristic function  
\begin{equation}\label{eq8.E}
F_N(t)=F_N(t;\nu;\epsi):=\prod_{i=1}^N\frac{t^2-(N-i+\epsi)^2}{t^2-(\nu_i+N-i+\epsi)^2}.
\end{equation}

\item

For $\ka\in\Sign^+_K$, we abbreviate $k_i:=\ka_i+K-i$, where $i=1,\dots,K$, and set
$$
d_K(\ka;\epsi):=\prod_{1\le i<j\le K}((k_i+\epsi)^2-(k_j+\epsi)^2).
$$
This agrees with the definition \eqref{eq1.d_N}. 

\end{itemize}

\begin{theorem}\label{thm8.A}
In the three distinguished cases \eqref{eq8.A} the following formula holds
\begin{equation}\label{eq8.B}
\frac{\La^N_K(\nu,\ka)}{d_K(\ka;\epsi)}=\det[M(i,j)]_{i,j=1}^K,
\end{equation}
where $[M(i,j)]$ is a $K\times K$ matrix whose entries are given by the following elementary expressions, which are in fact finite sums{\rm:}

$\bullet$ If $i<K$ or $i=K$ but $\ka_K>0$, then 
\begin{equation}\label{eq8.C}
M(i,j)=\sum_{m\ge k_i}\Res_{t=A_m}\big(g_{K-j}(t)F_N(t)\big) E(m,k_i).
\end{equation}

$\bullet$ If $i=K$ and $\ka_K=0$, then
\begin{equation}\label{eq8.D}
M(i,j)=M(K,j)=1+\sum_{m\ge 1}\Res_{t=A_m}\big(g_{K-j}(t)F_N(t)\big) E(m,0).
\end{equation}
\end{theorem}

\begin{proof}
By Theorem \ref{thm4.B}, 
$$
\frac{\La^N_K(\nu,\ka)}{d_K(\ka;\epsi)}=\det[(g_{k-j}F_N: g_{k_i})]_{i,j=1}^K.
$$
Next, recall that Proposition \ref{prop5.A} gives a summation formula for the transition coefficients $(\phi:g_k)$ (see \eqref{eq5.coeff}). We take $\phi=g_{k-j}F_N$ and $k=k_i$. Then \eqref{eq5.coeff} takes the form indicated in \eqref{eq8.C} and \eqref{eq8.D}. Due to Theorem \ref{thm6.A} and Theorem \ref{thm7.A}, we have elementary expressions for the functions $g_{K-j}(t)$ and the coefficients $E(m,k_i)$. The functions $F_N(t)$ are also given by an elementary expression (see \eqref{eq8.E}). 
\end{proof}

\subsection{Symplectic and orthogonal versions of the discrete B-spline}\label{sect8.2}

We write out in a more explicit form the formulas of Theorem \ref{thm8.A} in the particular case $K=1$. 

\begin{corollary}\label{cor8.A}
Let $\nu\in\Sign^+_N$ and $k\in\Sign^+_1=\{0,1,2,\dots\}$.

{\rm(i)} For $(a,\epsi)=(\tfrac12,1)$ {\rm(}the series $\CC$\,{\rm)},
\begin{gather*}
\La^N_1(\nu,k)=2(k+1)(N-1)(2N-1)\\
\times\sum_{i:\, \nu_i-i+1\ge k}\dfrac{(\nu_i-i+2-k)_{2N-3}}{(\nu_i+N-i+1)\prod\limits_{r: \, r\ne i}((\nu_i+N-i+1)^2-(\nu_r+N-r+1)^2)}, \quad k\ge1,\\
\La^N_1(\nu,0)=1-\sum_{i:\, \nu_i-i\ge0}\frac{2(\nu_i+4N-i-2)(\nu_i+N-i+1)}{(\nu_i+2N-i+1)(\nu_i+2N-i)(\nu_i+2N-i-1)}.
\end{gather*}

{\rm(ii)} For $(a,\epsi)=(\tfrac12,\tfrac12)$ {\rm(}the series $\BB$\,{\rm)},
\begin{gather*}
\La^N_1(\nu,k)=2(k+\tfrac12)(N-1)(2N-1)\\
\times\sum_{i:\, \nu_i-i+1\ge k}\dfrac{(\nu_i-i+2-k)_{2N-3}}{(\nu_i+N-i+\tfrac12)\prod\limits_{r: \, r\ne i}((\nu_i+N-i+\tfrac12)^2-(\nu_r+N-r+\tfrac12)^2)}, \quad k\ge1,\\
\La^N_1(\nu,0)=1-\sum_{i:\, \nu_i-i\ge0}\frac{2(\nu_i-i+1)+6N-5)}{(\nu_i+2N-i)(\nu_i+2N-i-1)}.
\end{gather*}

{\rm(iii)} For $(a,\epsi)=(-\tfrac12,0)$ {\rm(}the series $\DD$\,{\rm)},
\begin{gather*}
\La^N_1(\nu,k)=2(N-1)\,\sum_{i:\, \nu_i-i+1\ge k}\dfrac{(\nu_i-i+2-k)_{2N-3}}{\prod\limits_{r: \, r\ne i}((\nu_i+N-i)^2-(\nu_r+N-r)^2)}, \quad k\ge1,\\
\La^N_1(\nu,0)=1-2\sum_{i:\, \nu_i-i\ge0}\frac1{\nu_i+2N-i-1}.
\end{gather*}
\end{corollary}

\begin{proof}
In the case $K=1$, the formulas of Theorem \ref{thm8.A} are slightly simplified. Namely, the factor $d_K(\ka;\epsi)$ on the left-hand side of \eqref{eq8.B} disappears (because it becomes the empty product equal to $1$); the factor $g_{K-j}(t)$ on the right-hand side of \eqref{eq8.C} and \eqref{eq8.D} also disappears (because it turns into $g_0(t)\equiv1$.)  Taking this into account we obtain
$$
\La^N_1(\nu,k)=\begin{cases} \sum_{m\ge k}\Res_{t=A_m}F_N(t)\big) E(m,k), & k\ge1, \\
1+\sum_{m\ge 1}\Res_{t=A_m}\big(F_N(t)\big) E(m,0), & k=0.
\end{cases}
$$

Further, we have $A_m=N+\epsi+m-1$, because $K=1$ implies $L=N$. Then we substitute the values of $E(m,k)$ from Theorem \ref{thm7.A} and compute the residues of $F_N(t)$ from \eqref{eq8.E}, which is easy. This leads to the expressions displayed above. 
\end{proof}

Putting aside the end point $k=0$, we see that for $\nu$ fixed, $\La^N_1(\nu,k)$ is given by a piecewise polynomial function in $k$ of degree not depending on $\nu$ (the degree is $2N-2$ in type $\CC$ and $\BB$, and $2N-3$ in type $\DD$). Note that the structure of the formulas above is similar to that of the discrete B-spline, cf. \eqref{eq1.E} and \eqref{eq1.B}.  

Likewise, for $K\ge2$, the coefficients $\La^N_K(\nu,\ka)$ (where $K=2,\dots, N-1$) can be written as determinants of $K\times K$ matrices whose entries are one-dimensional piecewise polynomial functions. 

The appearance of piecewise polynomial expressions is not too surprising. Similar effects arise in other spectral problems of representation theory, such as weight multiplicities or decomposition of tensor products (see e.g. Billey-Guillemin-Rassart \cite{BGR}, Rassart \cite{Rassart}).  A specific feature of our problem, however, is that the description of a multidimensional picture can be expressed in terms of one-dimensional spline-type functions and, moreover, we end up with elementary formulas whose structure is much simpler than, say, that of Kostant's partition function for the weight multiplicities.

\section{Concluding remarks}\label{sect9}

\subsection{Contour integral representation}\label{sect9.1}

The large-$N$ limit transition mentioned in the introduction (section \ref{sect1.10}, item 5) relies on the possibility to represent the sums in  \eqref{eq8.C} and \eqref{eq8.D} as contour integrals. This contour integral representation is deduced from the next proposition, which is also of some independent interest. 

\begin{proposition}\label{prop9.A}
We keep to the assumptions and notation of section \ref{sect8.1}. Let $E(m,k)$ stand for  the transition coefficients given by the formulas of Theorem \ref{thm7.A}. We assume that $L\in\{2,3,\dots\}$ is fixed. 

{\rm(i)} There exists a function $R(t,k)$ of the variables $t\in\C$ and $k\in\{0,1,2,\dots\}$, such that $R_k(t):=R(t,k)$  is a rational function of  $t$ for each fixed value of $k$ and 
$$
E(m,k)=R(A_m,k), \qquad  m=1, 2,3,\dots\,.
$$

{\rm(ii)} These properties determine $R(t,k)$ uniquely. 

{\rm(iii)} For any $k=1,2,\dots$, the function $R_k(t)$ does not have poles in the right half-plane 
$$
\mathcal H(\epsi,L):=\{t\in\C: \Re t>L+\epsi-1\}. 
$$
\end{proposition}

We recall that 
\begin{equation}\label{eq9.B}
A_m=L+\epsi+m-1, \quad \text{where $\epsi=1,\tfrac12,0$}
\end{equation}

\begin{proof}
(i) To define $R(t,k)$, we use \eqref{eq9.B} as a prompt and simply replace $m$ by $t-L-\epsi+1$ in the formulas of Theorem \ref{thm7.A}. For $k=0$, we evidently get a rational function in $t$. For $k\ge1$, we write down the result by using the gamma function instead of the factorials (to distinguish between the series $\CBD$ we add the corresponding superscript, as before): 
$$
R^{(\CC)}(t,k)=2(k+1)(2L-2)(2L-1)\,\frac{t\,\Ga(t-L)\Ga(t+L-k-2)}
{\Ga(t-L-k+1)\Ga(t+L+1)}.
$$
$$
R^{(\BB)}(t,k)=2(k+\tfrac12)(2L-2)(2L-1)\,\frac{\Ga(t-L+\tfrac12)\Ga(t+L-k-\tfrac32)}
{\Ga(t-L-k+\tfrac32)\Ga(t+L+\tfrac12)}.
$$
$$
R^{(\DD)}(t,k)=2(2L-2)\,\frac{\Ga(t-L+1)\Ga(t+L-k-1)}
{\Ga(t-L-k+2)\Ga(t+L)}.
$$
From these expressions, the rationality property becomes also evident: we use the fact that  if $\al$ and $\be$ are two constants with $\al-\be\in\Z$, then the ratio  $\Ga(t+\alpha)/\Ga(t+\be)$ is a rational function in $t$.

(ii) The uniqueness claim is evident.

(iii) In each of the three variants, there are different ways to split the expression containing the 4 gamma functions into the product of two fractions of the form
$$
\frac{\Ga(t+\al_1)}{\Ga(t+\be_1)}\cdot \frac{\Ga(t+\al_2)}{\Ga(t+\be_2)},
$$
where $\al_1-\be_1$ and $\al_2-\be_2$ are integers. For our purpose, it is convenient to form the first fraction from the second $\Ga$-factor in the numerator and the firsr $\Ga$-factor in the denominator. Then the first fraction is a polynomial. As for the second fraction, it has the form
$$
\frac{\Ga(t-L-\epsi+1)}{\Ga(t+L+\epsi)}=\prod_{j=1}^{2L+2\epsi}\frac1{t-(L+\epsi-j)}
$$
and hence is regular in $\mathcal H(\epsi,L)$. 
\end{proof}

\subsection{A biorthogonal system of rational functions}\label{sect9.2}
Let $\FF^0(\epsi,L)\subset \FF(\epsi,L)$ denote the codimension $1$ subspace of functions vanishing at infinity. As in the previous subsection, we fix $L\ge2$ and assume that $\epsi$ takes one of the three values $1$, $\tfrac12$, $0$. Because $g_k(\infty)=1$, the functions $g^0_k(t):=g_k(t)-1$, where $k=1,2,\dots$, form a basis of $\FF^0(\epsi,L)$. The functions $R(t,k)$ and the half-plane $\mathcal H(\epsi,L)$ were introduced in Proposition \ref{prop9.A}.

\begin{proposition}
The two systems of rational functions,  $\{g_k^0(t):k=1,2,\dots\}$ and $\{R_k(t):=R(t,k): k=1,2,\dots\}$, are biorthogonal  
in the sense that
$$
\frac1{2\pi i}\oint_C g_\ell(t)R_k(t)dt=\de_{k\ell}, \qquad k,\, \ell=1,2,\dots,
$$
where as $C$ one can take an arbitrary simple contour in $\mathcal H(\epsi,L)$ with the property that it goes in the positive direction and encircles all the poles of $g_\ell(t)$. 
\end{proposition} 

\begin{proof}
The left-hand side is equal to the sum of the residues of the integrand. Because $R_k(t)$ is regular in $\mathcal H(\epsi,L)$, we may replace $g^0_\ell(t)$ by $g_\ell(t)$ and write this sum as 
\begin{equation}\label{eq9.A}
\sum_{m=1}^\infty (\Res_{t=A_m}(g_\ell(t))R_k(A_m)=\sum_{m=1}^\infty (\Res_{t=A_m}(g_\ell(t))E(m,k),
\end{equation}
where the equality holds because $R_k(A_m)=E(m,k)$ (see Proposition \ref{prop9.A}, item (i)). 

On the other hand, by the definition of the coefficients $E(m,k)=(e_m:g_k)$, for any function $f\in\FF(\epsi,L)$ one has 
\begin{equation*}
(f:g_k)=\sum_{m=1}^\infty (\Res_{t=A_m}f(t)) E(m,k), \qquad k=1,2,3, \dots\,.
\end{equation*}

Applying this to $f=g_\ell$ we conclude that \eqref{eq9.A} is equal to $(g_\ell:g_k)=\de_{k\ell}$, as desired.
\end{proof}

\subsection{Degeneration $g_k(t) \to \wt P^{(a,b)}_k(x)$}\label{sect9.4}
Let 
$$
\wt P^{(a,b)}_k(x)=\frac{P^{(a,b)}_k(x)}{P^{(a,b)}_k(1)}, \quad k=0,1,2,\dots,
$$
be the Jacobi polynomials with parameters $(a,b)$, normalized at the point $x=1$. Next, consider the rational functions $g_k(t;a,\epsi,L)$ given by the terminating hypergeometric series \eqref{eq1.g_k}. Recall that $\epsi=(a+b+1)/2$. We rescale $t=sL$, where $s$ is a new variable, which is related to $x$ via 
$$
x=\frac{s^2+1}{s^2-1}=\frac12\left(\frac{s+1}{s-1}+\frac{s-1}{s+1}\right).
$$
Under these assumptions, the following limit relation holds
$$
\lim_{L\to\infty}g_k(sL; a,\epsi,L)=\wt P^{(a,b)}_k(x), \quad k=0,1,2,\dots\,.
$$
It is easily verified from \eqref{eq1.g_k} and the expression of the Jacobi polynomials through the Gauss hypergeometric function. 

\subsection{A three-term recurrence relations for the functions $g_k(t)$}\label{sect9.3}

Wilson's thesis \cite{Wilson} contains a list of three-term recurrence relations satisfied by every terminated balanced hypergeometric series 
$$
F={}_4F_3\left[\begin{matrix}a,\, b,\, c,\, d
\\e,\,  f,\, g\end{matrix}\Biggl|1\right].
$$
One of them (formula (4.9) in \cite[p. 47]{Wilson}) reads as
\begin{multline*}
\frac{a(e-b)(f-b)(g-b)}{a-b+1}(F(a^+,b^-)-F)-\frac{b(e-a)(f-a)(g-a)}{b-a+1}(F(a^-,b^+)-F)\\
+cd(a-b)F=0.
\end{multline*}
This formula is applicable to the ${}_4F_3$ series defining the rational functions $g_k(t)$ (see \eqref{eq1.g_k}). It follows that the functions $g_k(t)$ satisfy a three-term recurrence relation, which is of the type investigated by Zhedanov \cite{Zhedanov-1999}.

\section*{List of symbols}

$\epsi$,  1.7. $\La^N_K$, 1.6. $\phi_{\nu,N}$, 1.5.  $(\phi:g_k)$, 1.8. $\si_{\mu,N}$, 2.1. $\AAA(\epsi,L)$, 1.8. $A_m$, 1.8. $C(N,\mu;a)$, 3. $d_N(\nu;\epsi)$, 1.7. $e_m$, 1.9. $E(m,k)$, 7. $f_\ell$, 5. $F_N$, 1.7.  $\FF(\epsi,L)$, 1.8. $\FF^0(\epsi,L)$, 9.2. $g_k(t)$, 1.7.  $G_{\ka,K}$, 1.7. $L$, 1.7. $P^{(a,b)}_{\nu,N}$, 1.6. $\wt P^{(a,b)}_{\nu,N}$, 1.6. $\Res_{t=A_m}$, 1.9. $R_k(t)$, 9.1. $R(k,t)$, 9.1. $S_{\mu,N}$, 2.1. $V_N$, 1.5.  $(x\mid c_0,c_1,\dots)^m$, 2.1.

\bigskip

Institute for Information Transmission Problems, Moscow, Russia.

Skolkovo Institute of Science and Technology, Moscow, Russia.

HSE University, Moscow, Russia.

\medskip
\emph{e-mail}: olsh2007@gmail.com

\end{document}